\documentclass[11pt]{amsart}
\usepackage[utf8]{inputenc}
\usepackage{amssymb,amsmath}
\usepackage{amsthm}
\usepackage{graphicx}
\usepackage{color}
\usepackage{tikz}
\usepackage{pgfplots}
\usepackage{enumitem}
\usepackage{algorithm}
\usepackage{setspace}
\usepackage{algpseudocode}
\usepackage{mathtools}
\usepackage{xfrac}
\usepackage{siunitx}

\usepackage{accents}
\usepackage{multicol} 
\usepackage{soul}
\usepackage{subcaption}
\usepackage{booktabs}
\usepackage{array}
\newcolumntype{L}[1]{>{\raggedright\let\newline\\\arraybackslash\hspace{0pt}}m{#1}}
\newcolumntype{C}[1]{>{\centering\let\newline\\\arraybackslash\hspace{0pt}}m{#1}}
\newcolumntype{R}[1]{>{\raggedleft\let\newline\\\arraybackslash\hspace{0pt}}m{#1}}
\usepackage{siunitx}
\usepackage{bbold}
\usepackage{comment}
\usepackage{float}

\pgfplotsset{compat=newest}

\usepackage[hidelinks]{hyperref}
\usepackage[nameinlink,noabbrev]{cleveref}
\newtheorem{theorem}{Theorem}
\newtheorem{proposition}[theorem]{Proposition}
\newtheorem{corollary}[theorem]{Corollary}
\theoremstyle{definition}

\newtheorem{lemma}[theorem]{Lemma}

\theoremstyle{remark}
\newtheorem{remark}[theorem]{Remark}

\newtheorem{assumption}[theorem]{Assumption}
\Crefname{assumption}{Assumption}{Assumptions}
\numberwithin{theorem}{section}
\numberwithin{equation}{section}
\numberwithin{table}{section}
\numberwithin{figure}{section}
\definecolor{myBlue}{RGB}{30,144,255} 
\definecolor{myGreen}{RGB}{69,169,0} 
\definecolor{myRed}{RGB}{165,12,42} 
\definecolor{myOrange}{RGB}{225,92,22} 
\def\N{\mathbb{N}}

\def\R{\mathbb{R}}

\def\W{\mathcal{W}} 
\def\VH{V_H}

\newcommand{\calC}{\ensuremath{\mathcal{C}} } 
\newcommand{\calT}{\ensuremath{\mathcal{T}} } 
\newcommand{\calB}{\ensuremath{\mathcal{B}} } 
\newcommand{\calBl}{\ensuremath{\mathcal{B}_\mathrm{lin}} } 

\def\ddiv{\mathrm{div}} 
\def\id{\mathbb{1}} 
\def\IH{\mathcal{I}_H} 
\def\Nb{\mathrm{N}} 
\def\supp{\mathrm{supp}} 

\newcommand{\eps}{\ensuremath{\varepsilon}}
\newcommand{\iu}{\mathbf{i}}

\def\vB{\overline{v}}

\def\vH{v_H}
\def\uH{u_H}
\def\uHT{\tilde{u}_H}

\makeatletter
\newenvironment{breakablealgorithm}
{%
	\begin{center}
		\refstepcounter{algorithm}%
		\hrule height.8pt depth0pt \kern2pt%
		\renewcommand{\caption}[2][\relax]{%
			{\raggedright\textbf{\ALG@name~\thealgorithm} ##2\par}%
			\ifx\relax##1\relax %
			\addcontentsline{loa}{algorithm}{\protect\numberline{\thealgorithm}##2}%
			\else %
			\addcontentsline{loa}{algorithm}{\protect\numberline{\thealgorithm}##1}%
			\fi
			\kern2pt\hrule\kern2pt
		}
	}{%
		\kern2pt\hrule\relax
	\end{center}
}
\makeatother

\allowdisplaybreaks
\begin{document}
\title[Multiscale Scattering in Kerr-Type Media]{Multiscale Scattering in Nonlinear Kerr-Type Media}
\author[Roland Maier, Barbara Verf\"urth]{Roland~Maier$^\dagger$, Barbara~Verf\"urth$^\ddagger$}
\address{${}^{\dagger}$ Department of Mathematical Sciences, Chalmers University of Technology and University of Gothenburg, 412 96 G\"oteborg, Sweden}
\email{roland.maier@chalmers.se}

\address{${}^{\ddagger}$ Institut f\"ur Angewandte und Numerische Mathematik, Karlsruher Institut f\"ur Technologie, Englerstr. 2, 76131 Karlsruhe}
\email{barbara.verfuerth@kit.edu}
\date{}
\begin{abstract}
We propose a multiscale approach for a nonlinear Helmholtz problem with possible oscillations in the Kerr coefficient, the refractive index, and the diffusion coefficient. The method does not rely on structural assumptions on the coefficients and combines the multiscale technique known as Localized Orthogonal Decomposition with an adaptive iterative approximation of the nonlinearity. We rigorously analyze the method in terms of well-posedness and convergence properties based on suitable assumptions on the initial data and the discretization parameters. Numerical examples illustrate the theoretical error estimates and underline the practicability of the approach.
\end{abstract}

\maketitle

	
\section{Introduction}
Wave propagation in heterogeneous and nonlinear media has arisen growing interest in the last years since corresponding materials can produce unusual effects, such as a negative refractive indices \cite{SmiPW04}, cloaking \cite{PenSS06}, or optical bistability \cite{GolG84}, to name a few.
Due to the resulting new effects, the deviations from the standard setup of homogeneous, linear materials come with a wide range of applications, such as perfect lenses \cite{Pen00} or mode-locking lasers \cite{ProWW93}. 
Heterogeneous materials occur, for instance, in the large field of metamaterials, i.e., artificially constructed composites \cite{JoaJWM08}.
Nonlinear material laws are required for instance for large intensities, when linearized models are no longer accurate enough.
One important example are Kerr-type media \cite{Ker75} in electromagnetics, where the electric permittivity depends on the electric field like $\varepsilon(E)=(\varepsilon_0+\varepsilon_2 |E|^2)$.

In the case of a monochromatic source, one can consider the problem in the time-harmonic regime.
Assuming additionally a linear transversal polarization of the electric field, one arrives at the nonlinear Helmholtz equation
\[-\ddiv (A \nabla u)  - k^2n(1 + \eps|u|^2)u \; = \; f,\]
where $u$ denotes the (transverse) component of the electric field, the material coefficients $A$, $n$, and $\varepsilon$ represent the inverse magnetic permeability, the linear part $\varepsilon_0$, and the nonlinear part $\varepsilon_2$ of the electric permittivity, respectively, and $k$ denotes the wave number.
The nonlinear Helmholtz equation is also an important model in nonlinear acoustics where $u$ represents the pressure.
In the following, we do not focus on the specific application and treat the nonlinear Helmholtz equation as a mathematical model with general material coefficients $A$, $n$, and $\varepsilon$.
Assumptions on these coefficients as well as appropriate boundary conditions are given further below.
In this contribution, 
the coefficients $A$, $n$, $\eps$ may vary on small spatial scales.
The nonlinear Helmholtz equation has been studied analytically and numerically for constant coefficients in \cite{EveW14,YuaL17,WuZ18} and in layered media in \cite{BarFT09,XuB10}.
However, the standard approximation tools (finite differences and finite elements) used in these works need to resolve all variations in the coefficients which ultimately leads to high-dimensional systems of linear equations and easily exceeds today's computational resources.
Therefore, computational multiscale methods are required, which deliver a macroscopic representation of the solution with drastically reduced computational effort.
Prominent examples include the Heterogeneous Multiscale Method  \cite{EE03,EE05,AbdEEV12}, the (Generalized) Multiscale Finite Element Method \cite{HowW97,EfeH09,EfeGH13}, the Multiscale Hybrid Mixed Method \cite{AraHPV13}, or the Localized Orthogonal Decomposition (LOD) \cite{MalP14,HenP13}.
The linear heterogeneous Helmholtz equation is studied, for instance, in \cite{ObeP98,CiaS14,OhlV18,GaoFC18,ChaV20}. 

The main contributions of this article are the presentation and numerical analysis of multiscale methods in the spirit of the LOD for the Helmholtz equation with Kerr-type nonlinearity.
Various works have successfully applied the LOD to wave propagation problems such as the wave equation \cite{AbdH17,PetS17,MaiP19,GeeM21}, the Helmholtz equation with constant \cite{GalP15,Pet17} and spatially varying coefficients \cite{BroGP17,PetV20} as well as time-harmonic Maxwell's equations \cite{GalHV17, Ver17, HenP20}. 
Besides dealing with multiscale coefficients, the LOD can also reduce the well-known pollution effect for the linear Helmholtz equation \cite{GalP15,Pet17}.
However, those strategies for heterogeneous wave propagation problems mostly rely heavily on linear arguments.
While the nonlinear Helmholtz equation is semilinear, we cannot treat it as a small perturbation of a linear diffusion problem as in \cite{HenMP12,HenMP13} because the wave number $k$ may be very large and dominate the behavior of the solutions.

In this work, we combine ideas on (iterative) finite element approximations for the nonlinear Helmholtz equation with constant coefficients~\cite{WuZ18} and on the construction of multiscale spaces known from the LOD for nonlinear problems by linearization \cite{Ver19}.
More specifically, we present and analyze iterative multiscale approximations based on a fixed-point iteration for the nonlinear Helmholtz equation.
We use an error indicator to locally decide in each step whether to update the multiscale basis.
For sufficiently small tolerance employed in this decision, we show an a priori error estimate which is of optimal order in the mesh size -- independent of the possible low regularity of the exact solution.
We need to take into account this low regularity, i.e., not more than $H^1$, when estimating the nonlinearity.
Since our analysis is largely based on fixed-point arguments, the above results hold under the assumption of sufficiently small data $\varepsilon$ and $f$. In particular, we require a~condition of the form $C_{\mathrm{stab},0}^3(k)C_nC_\eps k^{d-1}\|f\|_{L^2(D)} \leq c < 1$, where $C_{\mathrm{stab},0}(k)$ is the stability constant of the linear Helmholtz problem (i.e., \eqref{eq:stabLin0} with $\varepsilon = 0$) and $C_n$ and~$C_\eps$ are the upper bounds for $n$ and $\eps$, respectively; cf.~also \eqref{eq:stabRHS2} and~\eqref{eq:boundRHS} for the precise conditions. 
While such a condition is expected, its precise form does not seem to be sharp with respect to the numerical experiments. Lastly, we mention that we restrict ourselves to linear transversal polarization. The treatment of other polarizations and, in general, nonlinear coefficients $A$ requires substantial adjustments in the numerical analysis that go far beyond the scope of this article. 

The paper is organized as follows. In Section \ref{s:NLH}, we introduce our model problem and show the existence and uniqueness of solutions under minimal regularity assumptions. These results complement \cite{WuZ18} and may be of own interest.
Our multiscale approaches are introduced and analyzed in Section~\ref{s:multiscale}.
Finally, numerical experiments in Section~\ref{s:numex} illustrate our theoretical findings. 
More technical proofs of the main results are collected in the Appendix.\\[0.5em]
\paragraph{\textbf{Notation}.\ }
Unless otherwise mentioned, all our functions are complex-valued and we use $\overline{v}$ to denote the complex conjugate of $v$.
For any (sub)domain $S$, $(\cdot, \cdot)_S$ denotes the complex $L^2$-scalar product (with complex conjugate in the second argument).
Further, we use the following norms $\|\cdot\|_{0,S} := \|\cdot\|_{L^2(S)}$ and $|\cdot|_{1,S} := \|\nabla \cdot\|_{0,S}$.
As usual in the Helmholtz context, we also employ the following $k$-weighted norm $\|\cdot\|^2_{1,k,S} := |\cdot|_{1,S}^2 + k^2\,\|\cdot\|_{0,S}^2$ with associated scalar product $(\cdot, \cdot)_{1,k,S}$.
We will omit the subdomain $S$ in the notation of norms and scalar products if it equals the full computational domain $D$ and no confusion can arise.
Last, we use the notation $a \lesssim b$ to indicate that there exists a generic constant $C$ such that $a \leq C b$.

\section{Helmholtz Equation with Kerr-Type Nonlinearity}
\label{s:NLH}
\subsection{Model problem}
Let $D \subset \R^d$, $d \in \{2,3\}$ be a bounded domain with Lipschitz boundary $\Gamma = \partial D$ and outer normal $\nu$. 

In this work, we are interested in approximating the solution $u$ of the following nonlinear Helmholtz problem,
\begin{equation}\label{eq:NLH}
\begin{aligned}
-\ddiv (A \nabla u)  - k^2n(1 + \eps|u|^2)u \;& = \; f \quad\text{in }D,\\
\nabla u \cdot \nu + \iu k u \;& = \; 0 \quad\text{on }\Gamma,
\end{aligned}
\end{equation}
where $k$ is the wave number, $n$ the refractive index, and $ \eps$ the Kerr coefficient. Further, $A$ is the diffusion coefficient. Note that $\eps,\,A,$ and $n$ may depend on the spatial variable $x \in D$ and possibly vary on a fine scale.
For simplicity, we only consider scalar-valued material coefficients $A$, but the extension to matrix-valued coefficients is straightforward. Moreover, other types of boundary conditions could be studied as well. 
We make the following assumptions on the data throughout the whole article.
\begin{assumption}[]\label{a:coeff}
Suppose that 
\begin{itemize}
	\item $f\in L^2(D)$, 
	\item $k \geq k_0 > 0$, 
	\item $n,\, \eps,\, A\in L^\infty(D; \R)$ with $0 < c_n \leq n \leq C_n < \infty$, $0 \leq \eps \leq C_\eps < \infty$, and \mbox{$0 < c_A \leq A \leq C_A < \infty$} uniformly in $x$,
	\item $\supp (1-A)$, $\supp(1-n)$, and $\supp(\varepsilon)$ are compactly embedded in $D$.
\end{itemize}
\end{assumption}
Note that the last assumption ensures that $A=1$ and $n=1$ at the boundary $\Gamma$.
Since solutions to \eqref{eq:NLH} might not exist in the classical sense, we now consider the weak formulation of finding $u \in H^1(D)$ that solves
\begin{equation}\label{eq:NLHweak}
\calB(u,v) := (A \nabla u, \nabla v) -  (k^2n(1+\eps|u|^2) u,v) + \iu (k u,v)_\Gamma = (f,v)
\end{equation}
for all $v \in H^1(D)$.

The rest of this section is devoted to the existence and uniqueness of solutions to \eqref{eq:NLHweak}, see Theorem \ref{t:wellposedNL}.
The main idea is to approximate the solution $u$ by a fixed-point iteration, as already suggested in \cite{WuZ18}.
Due to the varying coefficients $A$ and $n$, however, some alterations have to be made. First, the well-posedness of the auxiliary linear problem, where the nonlinearity is \emph{fixed} (see the next section for a precise definition), is not clear from standard Helmholtz arguments.
Second, since we allow for general $L^\infty$-coefficients $A$ and $n$ we cannot necessarily expect solutions to linear Helmholtz problems to be in $H^2(D)$. Hence, one of the central arguments in \cite{WuZ18}, namely $L^\infty$-estimates for solutions to linear Helmholtz problems and the embedding of  $H^2(D)$ into $L^\infty(D)$ for $d\in\{2,3\}$ have to be replaced by arguments using only $H^1$-regularity.

\subsection{Auxiliary linear problem}
In view of linearization strategies that will be used to solve \eqref{eq:NLHweak}, we now introduce an auxiliary linearized version of \eqref{eq:NLHweak} which is characterized by $\calBl$ defined by
\begin{equation}
\calBl(\Phi;u,v) := (A \nabla u, \nabla v) - (k^2n\, u,v) - (k^2n\,\eps|\Phi|^2 u,v) + \iu  (k u,v)_\Gamma
\end{equation}
for $u,\,v,\,\Phi \in H^1(D)$. The operator $\calBl$ is sesquilinear with respect to the last two arguments. 
The auxiliary linear problem then reads: given $\Phi \in H^1(D)$, find $u_\Phi$ that solves 
\begin{equation}\label{eq:NLHweakLin}
\calBl(\Phi;u_\Phi,v) = (f,v)
\end{equation}
for all $v \in H^1(D)$. 
Note that $\calBl$ is bounded with respect to the second and third argument in the norm $\|\cdot\|_{1,k}$ provided that $k^2n\,\eps|\Phi|^2$ can be bounded, see Proposition~\ref{prop:auxlinpb}.

By setting $\Phi \equiv 0$, we obtain the \emph{classical (linear) Helmholtz problem} that  consists in finding $u_0\in H^1(D)$ such that
\begin{equation}\label{eq:NLHweakLin0}
\calBl(0;u_0,v) = (f,v)
\end{equation}
for all $v \in H^1(D)$.
If the unique continuation principle holds, \emph{Fredholm's alternative} can be employed to show that \eqref{eq:NLHweakLin0} possesses a unique solution.
As discussed in \cite{GraPS19,GraS20} in detail, Assumption \ref{a:coeff} is sufficient for the unique continuation principle to hold for $d=2$, whereas for $d=3$ one has to make additional assumptions, e.g., $A\in C^{0,1}(D)$.
Well-posedness of \eqref{eq:NLHweakLin0} via Fredholm's alternative, however, does not provide a quantitative stability estimate, in particular with a wave number explicit stability constant, for the solution $u_0$ of~\eqref{eq:NLHweakLin0}.
Here, we base our analysis on the following assumption.

\begin{assumption}[Well-posedness and stability of the classical Helmholtz problem]\label{a:stabLin0}
We assume that \eqref{eq:NLHweakLin0} possesses a unique solution and that there exists a constant $C_{\mathrm{stab},0}(k) > 0$ such that the solution $u_0 \in H^1(D)$ of \eqref{eq:NLHweakLin0} fulfills the stability estimate 
\begin{equation}\label{eq:stabLin0}
\|u_0\|_{1,k} \leq C_{\mathrm{stab},0}(k)\,\|f\|_0.
\end{equation}
\end{assumption}

Note that several works consider the dependence of $C_{\mathrm{stab},0}(k)$ on the wave number $k$, also in the present setup of heterogeneous coefficients $A$ and $n$; see, e.g., \cite{BroGP17,GraPS19,MoiS19,SauT18} and the references therein.
For instance, \cite{GraPS19} proves that $C_{\mathrm{stab},0}(k)\lesssim 1$ under certain conditions on the Lipschitz coefficients $A$ and $n$. If the involved coefficients are constant on some arbitrarily small open set, the reverse statement is also true, i.e., $C_{\mathrm{stab},0}(k)\gtrsim 1$; cf.~\cite[Lem.~5.5]{ChaGNT18}. 
Generally, the crucial point is the existence of so-called \emph{trapped rays} and a possible exclusion of certain frequencies, see the discussion and references in~\cite{GraPS19,LafSW20}. In the presence of trapped rays, \cite{LafSW20} provides a polynomial stability estimate for very general settings in the sense of $C_{\mathrm{stab},0}(k)\lesssim k^q$ for almost all $k \geq k_0$ with some $k_0 > 0$ and $q \geq 0$. If trapped rays are excluded completely, one can show that $C_{\mathrm{stab},0}(k)\lesssim 1$. We emphasize that we often have such a scaling in mind when considering the stability constants throughout this paper.

As a next step, we quantify the well-posedness of the auxiliary linear problem~\eqref{eq:NLHweakLin} based on the well-posedness of the classical Helmholtz problem. Therefore, we require a \emph{Nirenberg-type inequality} which is stated in the following lemma. 

\begin{lemma}[Nirenberg-type inequality]\label{l:Nir}
There exists a constant $C_\mathrm{Nir} > 0$ such that
\begin{equation*}\label{eq:Nir}
k^{1-d/3}\,\|v\|_{L^6(S)}\leq C_\mathrm{Nir}\,\|v\|_{1,k,S}.
\end{equation*}
for all $v \in H^1(D)$ and $S \subset D$. 
\end{lemma}

\begin{proof}
From \cite{Nir59}, there exists a constant $C_\mathrm{Nir} > 0$ such that
\begin{equation*}
\|v\|_{L^6(S)}\leq C_\mathrm{Nir}\, |v|_{1,S}^{d/3}\,\|v\|_{0,S}^{1-d/3}
\end{equation*}
for all $v \in H^1(S),\, S \subset D$. 
Since $|v|_{1,S} \leq \|v\|_{1,k,S}$ and $k\,\|v\|_{0,S} \leq \|v\|_{1,k,S}$, we obtain
\begin{equation*}
k^{1-d/3}\,\|v\|_{L^6(S)}\leq C_\mathrm{Nir}\, |v|_{1,S}^{d/3}\,k^{1-d/3}\,\|v\|_{0,S}^{1-d/3} \leq C_\mathrm{Nir}\,\|v\|_{1,k,S}. \qedhere
\end{equation*}
\end{proof}

\begin{proposition}\label{prop:auxlinpb}
Let $\Phi \in H^1(D)$ such that
\begin{equation}\label{eq:stabPhicontr}
C_{\mathrm{stab},0}(k) C_n C_\eps C^{3}_\mathrm{Nir}k^{d-1}\|\Phi\|_{1,k}^2\leq \vartheta<1.
\end{equation}
Then, there exists a unique solution~$u_\Phi$ of~\eqref{eq:NLHweakLin} and the stability estimate
\begin{equation}\label{eq:stabLin}
\|u_\Phi\|_{1,k} \leq C_{\mathrm{stab}}(k)\,\|f\|_0 
\end{equation}
holds with $C_\mathrm{stab}(k) = (1-\vartheta)^{-1}C_{\mathrm{stab},0}(k)$. 

Further, $\calBl(\Phi; \cdot, \cdot)$ is bounded with a continuity constant
\begin{equation*}
C_\calB \,\leq\, 2\,\max\{C_A,C_n\}+(1+2\, C_\mathrm{tr}),
\end{equation*} 
where $C_\mathrm{tr}$ denotes the constant in the trace inequality~\eqref{eq:trineq}. 
\end{proposition}
\begin{proof} 
Let $\Phi\in H^1(D)$. We define the operator $T\colon H^1(D)\to H^1(D)$, which maps \mbox{$\psi\in H^1(D)$} to the solution $T \psi\in H^1(D)$ of the following Helmholtz problem
\begin{equation*}
\calBl(0;T \psi,v) = (f+k^2n\eps|\Phi|^2 \psi,v). 
\end{equation*}
The solution to the auxiliary linear problem~\eqref{eq:NLHweakLin} can then equivalently be formulated as the fixed-point problem $u_\Phi=Tu_\Phi$.
The existence and uniqueness of the solution $u_\Phi$ of the auxiliary linear problem therefore follow from \emph{Banach's fixed-point theorem} (see, e.g., \cite[Thm.~5.7]{Bre11}) provided that $T$ is a contraction.

To show the contraction property, let $\psi_1,\, \psi_2\in H^1(D)$. We set~$w:=T\psi_1-T\psi_2$ and observe that $w$ solves
\begin{equation*}
\calBl(0;w,v) = (k^2n\eps|\Phi|^2 (\psi_1-\psi_2),v)
\end{equation*}
for all $v \in H^1(D)$.
With~\eqref{eq:stabLin0}, a generalized H\"older inequality (see, e.g., \cite[Rem.~2 in Chap.~4.2]{Bre11}), and Lemma~\ref{l:Nir}, we deduce
\begin{align*}
\|T\psi_1-T\psi_2\|_{1,k}&\leq C_{\mathrm{stab},0}(k)\|k^2 n\eps |\Phi|^2(\psi_1-\psi_2)\|_0\\
&\leq C_{\mathrm{stab},0}(k)k^2 C_nC_\eps\|\Phi\|^2_{L^6(D)}\|\psi_1-\psi_2\|_{L^6(D)}\\
&\leq C_{\mathrm{stab},0}(k)C_nC_\eps C^{3}_\mathrm{Nir} k^{d-1}\|\Phi\|_{1,k}^2\|\psi_1-\psi_2\|_{1,k}.
\end{align*}
The assumption \eqref{eq:stabPhicontr} yields the desired contraction and the formula for $C_\mathrm{stab}(k)$ immediately follows by means of a geometric series.

Regarding the continuity constant $C_\calB$, we use the following \emph{trace inequality} (see, e.g., \cite[Sec.~1.5]{Gri85}),
\begin{equation}\label{eq:trineq}
\|u\|^2_{0,\Gamma} \leq C_\mathrm{tr}\|u\|_0\|u\|_{1}.
\end{equation}
From this and using a weighted Young's inequality, we obtain
\begin{equation}\label{eq:proofCB1}
\begin{aligned}
|\iu (ku,v)| &\leq C_\mathrm{tr} k \|u\|^{1/2}_0\|u\|^{1/2}_{1} \|v\|^{1/2}_0\|v\|^{1/2}_{1}\\
&\leq C_\mathrm{tr} (k \|u\|_0 + \|u\|_1)(k \|v\|_0 + \|v\|_1)\\
&\leq 2\,C_\mathrm{tr} \|u\|_{1,k}\|v\|_{1,k}.
\end{aligned}
\end{equation}
With Lemma~\ref{l:Nir} and \eqref{eq:stabPhicontr}, we further have the rough estimate
\begin{equation}\label{eq:proofCB2}
|(k^2n\,\eps|\Phi|^2 u,v)| \leq k^2 C_nC_\eps\|\Phi\|^2_{L^6(D)}\|u\|_{L^6(D)}\|v\|_{0} \leq \vartheta \|u\|_{1,k}\|v\|_{1,k}.
\end{equation}
With \eqref{eq:proofCB1} and \eqref{eq:proofCB2}, we finally compute
\begin{align*}
|\calBl(\Phi;u,v)| &= |(A \nabla u, \nabla v) - (k^2n\, u,v) - (k^2n\,\eps|\Phi|^2 u,v) + \iu  (k u,v)_\Gamma| \\
&\leq C_A |u|_1 |v|_1 + C_n k^2 \|u\|_0 \|v\|_0 
+ (\vartheta + 2 C_\mathrm{tr})\|u\|_{1,k}\|v\|_{1,k}\\
&\leq \big(2\,\max\{C_A,C_n\}+(1+2\, C_\mathrm{tr})\big) \|u\|_{1,k}\|v\|_{1,k}.
\qedhere
\end{align*}
\begin{remark}\label{r:adjoint}
With the same techniques, one can as well show the existence and uniqueness of the solution to the adjoint problem.
\end{remark}
\end{proof}

If the above assumptions on $\Phi$ are satisfied, we can deduce that $\calBl(\Phi;\cdot,\cdot)$ fulfills an inf-sup condition on $H^1(D)$ as quantified in the next lemma. The result is used for the well-posedness of the discrete method in Section~\ref{s:multiscale} below.
\begin{lemma}[Inf-sup condition on $H^1(D)$]\label{l:infsup}
Assume that~\eqref{eq:stabPhicontr} is satisfied. Then, we have that 
\begin{equation}\label{eq:infsup}
\adjustlimits\inf_{v \in H^1(D)}\sup_{w \in H^1(D)}\frac{\Re\,\calBl(\Phi;v,w)}{\|v\|_{1,k}\,\|w\|_{1,k}} \geq \delta(k)
\end{equation}
with $\delta(k) = \frac{\min\{c_A, c_n\}}{C_{\mathrm{stab}}(k)(2\, kC_n+1)}$.
\end{lemma}
\begin{proof}
The proof follows \cite[Lem.~2.1]{Pet17}. The ideas, however, date back to \cite[Prop.~8.2.7]{Mel95} and \cite[Lem.~3.3 and 3.4]{ChaM08}. Let $v \in H^1(D)$ be given and define $z \in H^1(D)$ as the solution to
\begin{equation*}
\calBl(\Phi;w,z) = (k^2n(2+\eps|\Phi|^2)w,v)
\end{equation*}
for all $w \in H^1(D)$.
From Proposition~\ref{prop:auxlinpb}, Remark \ref{r:adjoint}, and Lemma~\ref{l:Nir}, we know that $z$ exists, is unique, and satisfies
\begin{equation*}
\begin{aligned}
\|z\|_{1,k} &\leq C_\mathrm{stab}(k)\big(2\,C_nk^2\|v\|_0 + C_nC_\eps C_\mathrm{Nir}^3k^{d-1}\|\Phi\|^2_{1,k}\|v\|_{1,k}\big)\\
&\leq C_\mathrm{stab}(k)(2\,kC_n+1) \|v\|_{1,k}.
\end{aligned}
\end{equation*}
We set $w = v + z$. The assertion then follows from the inequality
\begin{equation*}
\Re\,\calBl(\Phi;v,w) \geq \min\{c_A, c_n\}  \|v\|_{1,k}^2. \qedhere
\end{equation*}
\end{proof}

\subsection{Existence and stability of solutions to the nonlinear problem}

Based on the iterative procedure used in~\cite{WuZ18}, we now show existence and stability results for the nonlinear problem~\eqref{eq:NLHweak} based on the auxiliary problem~\eqref{eq:NLHweakLin} and the stability property quantified in Proposition~\ref{prop:auxlinpb}.

Let $u^0 \in H^1(D)$. Employing the linearized Helmholtz problem~\eqref{eq:NLHweakLin}, we consider the sequence of solutions $u^m \in H^1(D)$, $m \in \N$, which solve the sequence of problems given by
\begin{equation}\label{eq:NLHweakIt}
\calBl(u^{m-1};u^m,v) = (f,v).
\end{equation}
As a first step, we show that if condition \eqref{eq:stabPhicontr} holds for $\Phi = u^0$, it also holds for $u^m$, $m \in \N$, such that the stability estimate~\eqref{eq:stabLin} is valid for the whole sequence $\{u^m\}_{m \in \N}$. 
\begin{lemma}[Stability of iterative solutions]\label{l:stabIt}
Let $u^0\in H^1(D)$ such that~\eqref{eq:stabPhicontr} is fulfilled for $\Phi = u^0$. Further, suppose that 
\begin{equation}\label{eq:stabRHS}
C_{\mathrm{stab},0}(k)C^2_{\mathrm{stab}}(k)C_nC_\eps C^{3}_\mathrm{Nir}\,k^{d-1}\,\|f\|_0^2 \leq \vartheta. 
\end{equation}
Then, the sequence $\{u^m\}_{m \in \N}$ defined by \eqref{eq:NLHweakIt} fulfills the stability property
\begin{equation}\label{eq:stabum}
\|u^m\|_{1,k} \leq C_\mathrm{stab}(k)\,\|f\|_0 
\end{equation}
for all $m\in \N$ with the constant $C_\mathrm{stab}(k)$ from Proposition~\ref{prop:auxlinpb}.
\end{lemma}
Note that $C_{\mathrm{stab}}(k)= (1-\vartheta)^{-1}C_{\mathrm{stab}, 0}(k)$ in Proposition \ref{prop:auxlinpb} so that in fact \eqref{eq:stabRHS} is equivalent ot the following condition
\begin{equation*}
C_{\mathrm{stab},0}(k)^3C_nC_\eps C^{3}_\mathrm{Nir}\,k^{d-1}\,\|f\|_0^2 \leq \vartheta(1-\vartheta)^2.
\end{equation*} 
\begin{proof}[Proof of Lemma~\ref{l:stabIt}]
With~\eqref{eq:stabRHS} and~\eqref{eq:stabum} for some fixed $m \in \N$, we have that
\begin{equation*}
C_{\mathrm{stab},0}(k) C_n C_\eps C^{3}_\mathrm{Nir}\,k^{d-1}\,\|u^{m+1}\|_{1,k}^2
\leq C_{\mathrm{stab},0}(k)C^2_\mathrm{stab}(k) C_n C_\eps C^{3}_\mathrm{Nir}\,k^{d-1}\,\|f\|_0^2 
\leq \vartheta.
\end{equation*}
The assertion thus follows by induction using Proposition~\ref{prop:auxlinpb}.
\end{proof}

As a next step, we use the sequence $\{u^m\}_{m \in \N}$ to show existence and uniqueness of the solution $u \in H^1(D)$ of \eqref{eq:NLHweak}. 

\begin{theorem}[Well-posedness of the nonlinear Helmholtz problem]\label{t:wellposedNL}
Suppose the following slightly stronger version of~\eqref{eq:stabRHS} holds,
\begin{equation}\label{eq:stabRHS2}
2\,C^3_{\mathrm{stab}}(k)C_nC_\eps C^{3}_\mathrm{Nir}\,k^{d-1}\,\|f\|_0^2 \leq \vartheta. 
\end{equation} 
Then, there exists a unique solution $u \in H^1(D)$ of \eqref{eq:NLHweak}. {Further, we have} the stability estimate
\begin{equation}\label{eq:stabu}
\|u\|_{1,k} \leq C_\mathrm{stab}(k)\,\|f\|_0. 
\end{equation}
\end{theorem}
Before we prove the theorem, some remarks on the assumptions are in order.
We can again use $C_{\mathrm{stab}}(k)=(1-\vartheta)^{-1}C_{\mathrm{stab},0}(k)$ to equivalently write \eqref{eq:stabRHS2} in the following form
\[2\,C^3_{\mathrm{stab}, 0}(k)C_nC_\eps C^{3}_\mathrm{Nir}\,k^{d-1}\,\|f\|_0^2 \leq \vartheta(1-\vartheta)^3.\]
This is  sometimes termed as \emph{smallness of the data assumption} because it requires  the combination of wave number, refractive index, Kerr coefficient, and source term to be sufficiently small. 
Let us further elaborate on the influence of this condition in an ideal setting with $C_{\mathrm{stab},0}^3(k)C^{3}_\mathrm{Nir} \approx 1$. In this case, the assumption states that $n$, $\eps$, and/or $f$ need to be decreased in a reasonable way for large wavenumbers $k$. For fixed $n$ and $f$ and in the case $d=2$, the conditions leads to a relation $\eps \lesssim k^{-1}$ and therefore the nonlinear factor $nk^2\eps$ in~\eqref{eq:NLHweak} is allowed to be of order $k$. On the other hand, for $d=3$, we obtain $\eps \lesssim k^{-2}$ such that the term $nk^2\eps$ remains roughly constant when $k$ is increased. If the stability constants are not of order one, however, the nonlinear factor $nk^2\eps$ most likely needs to be decreased for growing $k$.
Note that the structure of the smallness condition is of course very similar to \cite{WuZ18} and the main difference is that we have a factor $k^{d-1}$ instead of the (better) factor $k^{d-2}$ in \cite{WuZ18}.
This comes from our different proof technique avoiding $L^\infty$-estimates and resorting to the Nirenberg inequality~\eqref{eq:Nir} instead.
\begin{proof}[Proof of Theorem \ref{t:wellposedNL}]
The proof follows the ideas of \cite[Thm.~2.5]{WuZ18}. Let $\{u^m\}_{m \in \N}$ be the sequence of solutions defined in \eqref{eq:NLHweakIt} starting from some $u^0\in H^1(D)$ satisfying~\eqref{eq:stabPhicontr}. We set $w^m := u^{m+1} - u^m$ and observe that $w^m$ solves
\begin{equation*}
\calBl(u^m;w^m,v) = (k^2n\eps(|u^m|^2 - |u^{m-1}|^2) u^m,v)
\end{equation*}
for all $v \in H^1(D)$. 
Since \eqref{eq:stabRHS2} implies \eqref{eq:stabRHS}, we obtain with Lemma~\ref{l:stabIt} and Lemma~\ref{l:Nir}
\begin{equation*}
\begin{aligned}
\|w^m\|_{1,k} &\leq C_\mathrm{stab}(k)\left\|k^2n\eps \big(|u^m|^2 - |u^{m-1}|^2\big) u^m\right\|_{0}\\
&\leq C_\mathrm{stab}(k)\,k^2C_nC_\eps\,\|u^m\|_{L^6(D)}\,\big(\|u^m\|_{L^6(D)} + \|u^{m-1}\|_{L^6(D)}\big)\,\|w^{m-1}\|_{L^6(D)}\\
&\leq 2\,C^3_\mathrm{stab}(k)C_nC_\eps C^{3}_\mathrm{Nir}\,k^{d-1}\,\|f\|_0^2\,\|w^{m-1}\|_{1,k}\\
&\leq \vartheta \|w^{m-1}\|_{1,k}.
\end{aligned}
\end{equation*}
This contraction property of $w^m$ yields $\|w^m\|_{1,k}\leq \vartheta^m\|w^0\|_{1,k}$.
By the (lower) triangle inequality, this implies -- as in the proof of Banach's fixed-point theorem \cite[Thm.~5.7]{Bre11} -- that $\{u^m\}_{m \in \N}$ is a Cauchy sequence with respect to $\|\cdot\|_{1,k}$ and converges to a~limit $u:=\lim_{m \to \infty} u^m \in H^1(D)$ which solves \eqref{eq:NLHweak}. The stability estimate \eqref{eq:stabu} directly follows from \eqref{eq:stabum}.

To show uniqueness, let $u$ and $\hat{u}$ be two solutions of~\eqref{eq:NLHweak}. Then, $w := u - \hat{u}$ solves 
\begin{equation*}
\calBl(u;w,v) = (k^2n\eps(|u|^2 - |\hat{u}|^2) \hat{u},v)
\end{equation*}
for all $v \in H^1(D)$. As above, we thus get
\begin{equation}
\|w\|_{1,k} \leq \vartheta \|w\|_{1,k},
\end{equation}
which implies $w = 0$ because of $\vartheta <1$.
\end{proof}
\begin{remark}\label{rem:fineFEM}
Since we do not exploit any higher regularity of solutions, the procedure in this section can verbatim be used to show existence and uniqueness of discrete solutions.
More precisely, let $V\subset H^1(D)$ be a closed subspace and further, let Assumption~\ref{a:stabLin0} be satisfied in $V$ and assume that~\eqref{eq:stabRHS2} holds accordingly. With an appropriate initial iterate $v^0 \in V$ (e.g., $v^0 = 0$) and the arguments in Proposition~\ref{prop:auxlinpb}, we directly obtain the existence and uniqueness of a Galerkin solution $v\in V$ to the Kerr-Helmholtz problem~\eqref{eq:NLHweak}.
Note, however, that the stability constants $C_{\mathrm{stab},0}(k)$ and thus $C_{\mathrm{stab}}(k)$ may generally depend on the subspace $V$.
\end{remark}
\begin{remark}\label{rem:Newton}
Let us emphasize that the linearization technique used above -- that will also be employed in the following -- is not the only possibly choice. The nonlinearity could also be dealt with methods such as a modified Newton's iteration as investigated, e.g., in~\cite{YuaL17} in the homogeneous setting.
\end{remark}
\section{Multiscale Approximations}
\label{s:multiscale}

In this section, we are concerned with the approximation of the solution to~\eqref{eq:NLHweak} in a finite-dimensional subspace. Since the present setting involves possible fine oscillations in the coefficients $A$, $n$, and $\eps$, a classical finite element approximation requires a resolution of any fine-scale features in order to provide reasonable approximations; see, e.g., \cite{BroGP17,PetV20} in the context of the linear Helmholtz problem.
Additionally, discretizations of the Helmholtz problem are subject to the so-called \emph{pollution effect} (see \cite{BabS97} and the references therein), so that for the lowest order finite element method the condition \mbox{$k^3h^2\lesssim 1$} typically has to be satisfied before convergence of the error is observed; see, e.g., \cite{BayGT85,AziKS88}.
Especially in the nonlinear setting where an iterative scheme is to be used, the required resolution of fine-scale features and the pollution effect lead to unfeasibly expensive computations. 

The multiscale construction that is presented in this section aims at resolving this issue by constructing appropriate approximation spaces on a coarse-scale level. The approach is based on the Localized Orthogonal Decomposition method, which was introduced in~\cite{MalP14} and further developed in~\cite{HenP13} for an elliptic model problem.
The coarse-scale level is characterized by the mesh size $H$ of a shape regular and quasi uniform quadrilateral/hexahedral mesh $\mathcal{T}_H$. This mesh is coarse in the sense that it does not resolve the fine-scale features of the coefficients $A$, $n$, and $\varepsilon$.
We work with quadrilateral/hexahedral meshes in the following, but we emphasize that all arguments and results carry over to simplicial meshes directly.
Denote by $Q_1(\calT_H)$ the space of possibly discontinuous functions that are polynomials of coordinate degree at most one on each element of~$\calT_H$.
We set $V_H:=Q_1(\calT_H)\cap H^1(D)$ the coarse Lagrange finite element space.
The LOD approach is based upon a~so-called quasi-interpolation operator $\IH\colon H^1(D) \to \VH$ with the following properties
\begin{align}
H^{-1}\,\|(\operatorname{Id} - \IH)v\|_{0,T} + |\IH v |_{1,T} &\leq \tilde C_\mathrm{int}\, |v|_{1,\Nb(T)}, \quad v \in H^1(D),\label{eq:IH1loc}\\ 
\|\IH v\|_{0,T} &\leq \tilde C_\mathrm{int}\, \|v\|_{0,\Nb(T)}, \quad v \in L^2(D), \label{eq:IH2loc}\\
\IH \circ \IH &= \IH, \label{eq:IH3}
\end{align}
where $\Nb(T)$ denotes the neighborhood of the element $T \in \calT_H$ defined by
\begin{equation*}
\Nb(T) := \bigcup\bigl\{K\in\calT_H\,\colon\,\overline{K}\cap\overline{T} \neq\emptyset\bigr\}.
\end{equation*}
The particular choice that we use in our numerical experiments is $\IH := \pi_H \circ \Pi_H$, where $\Pi_H$ is the piecewise $L^2$-projection onto $Q_1(\calT_H)$. Moreover, $\pi_H$ denotes an averaging operator that, for any $\vH\in Q_1(\calT_H)$ and any vertex $z$ of $\calT_H$, is characterized by
\begin{equation*}
\big(\pi_H(\vH)\big)(z) := \sum_{K\in\calT_H:\atop z \in \overline{K}} \big({\vH\vert}_K\big)(z)\cdot\frac{1}{\mathrm{card}\{T \in \calT_H\,\colon\, z \in \overline{T}\}}.
\end{equation*}
This choice of $\IH$ satisfies the properties \eqref{eq:IH1loc}--\eqref{eq:IH3}. We refer to \cite{Osw93,Bre94,ErnG17} for a proof of these conditions.
Note that from \eqref{eq:IH1loc}--\eqref{eq:IH2loc}, we can directly derive the following estimates on the whole domain $D$,
\begin{align}
H^{-1}\,\|(\operatorname{Id} - \IH)v\|_{0} + |\IH v |_{1} &\leq C_\mathrm{int}\, |v|_{1}, \quad v \in H^1(D),\label{eq:IH1}\\
\|\IH v\|_{0} &\leq C_\mathrm{int}\, \|v\|_{0}, \quad v \in L^2(D). \label{eq:IH2}
\end{align}

Based on the operator $\IH$, we define the so-called \emph{fine-scale space} $\W$ as its kernel with respect to $H^1$-functions, i.e.,
\begin{equation*}
\W := \ker {\IH\vert}_{H^1(D)}.
\end{equation*}
Next, we define an auxiliary corrector problem based on a function $\Phi$ that fulfills \eqref{eq:stabPhicontr} as follows. Let $\Phi \in H^1(D)$ be given and define the \emph{correction operator} $\calC_\Phi\colon H^1(D) \to \W$ for any $v \in H^1(D)$ as the solution to 
\begin{equation}\label{eq:corrPhi}
\calBl(\Phi;\calC_\Phi v, w) = \calBl(\Phi; v, w)
\end{equation}
for all $w \in \W$. 
Similarly, we also define the \emph{adjoint correction operator} $\calC^*_\Phi\colon H^1(D) \to \W$ for any $v \in H^1(D)$ as the solution to 
\begin{equation}\label{eq:corrPhi*}
\calBl(\Phi;w,\calC^*_\Phi v) = \calBl(\Phi;w, v)
\end{equation}
for all $w \in \W$ and remark that $\calC^*_\Phi v = \overline{\calC_\Phi \overline{v}}$. 

Note that \eqref{eq:corrPhi} and \eqref{eq:corrPhi*} are well-defined by the coercivity condition that is proved in the following lemma.
\begin{lemma}[coercivity condition on $\W$]\label{l:infsupW}
Let $\Phi \in H^1(D)$ fulfill~\eqref{eq:stabPhicontr}.
Assume that
\begin{equation}\label{eq:res}
kH \leq C_{\mathrm{res}}:=\frac{\sqrt{c_A}}{2C_\mathrm{int} \sqrt{C_n}}
\end{equation}
and 
\begin{equation}\label{eq:kerr}
2\,C_\mathrm{int}^{2-d/3}C_\mathrm{res}^{1-d/3}\,\vartheta H \leq C_{\mathrm{stab},0}(k)\,c_A.
\end{equation}
Then, it holds that 
\begin{equation}\label{eq:infsupW}
\Re\, \calBl(\Phi;w,w) \geq \gamma\,\|w\|^2_{1,k}
\end{equation}
for all $w \in \W$ with $\gamma:=\frac{c_A C_n}{4C_n+c_A}$.
\end{lemma}
Let us comment on the assumptions of Lemma~\ref{l:infsupW}. First, \eqref{eq:res} requires $kH$ to be sufficiently small (of the order one), which is a natural resolution condition because one always needs some degrees of freedom per wave length to faithfully represent the wave.
Second, we emphasize that \eqref{eq:kerr} does \emph{not} result in a (notable) restriction of the mesh size in particular for growing wave numbers: the left-hand side factors $C_\mathrm{int}^{2-d/3}C_\mathrm{res}^{1-d/3}\vartheta$ are independent of $k$ and of order one. As discussed in Section~\ref{s:NLH}, the term $C_{\mathrm{stab}, 0}(k)$ on the right-hand side fulfills $C_{\mathrm{stab}, 0}(k) \gtrsim 1$ in most cases and often scales polynomially in $k$ for large enough wave numbers. Therefore, it remains constant or is even growing with increasing $k$.
Hence, \eqref{eq:kerr} is no resolution condition that requires $H$ to become smaller for large frequencies. In fact, \eqref{eq:res} will in practice, especially for large $k$, be the dominating and important condition.
Finally, we mention that the dependency of \eqref{eq:res} and \eqref{eq:kerr} on $c_A$ may be removed by the use of $A$-weighted norms and suitable $A$-weighted interpolation operators, which is relevant in the high contrast case, where $c_A$ might be very small. This, however, is not the focus of the present work and we refer to \cite{PetV20}, for instance, for details in the context of the linear Helmholtz problem. 
\begin{proof}[Proof of Lemma \ref{l:infsupW}]
By the definition of $\W = \ker {\IH\vert}_{H^1(D)}$ and \eqref{eq:IH1}, we have for any $w \in \W$
\begin{equation}\label{eq:proof1}
\|w\|_0 = \|(\operatorname{Id} - \IH)w\|_0 \leq C_\mathrm{int}\,H\,|w|_1. 
\end{equation}
From~\eqref{eq:proof1} and~\eqref{eq:res}, we also get that $|\cdot|_{1}$ is a norm on $\W$ which is equivalent to the full norm $\|\cdot\|_1$ as well as the energy norm $\|\cdot\|_{1,k}$ with constants that are independent of $k$ and $H$. In particular, for any $w \in \W$, 
\begin{equation}\label{eq:equiv1}
|w|_1 \leq \|w\|_{1} \leq (1+C_\mathrm{int}^2H^2)^{1/2} \,|w|_{1} \leq (1+C_\mathrm{int}^2)^{1/2} \,|w|_{1}
\end{equation}
and
\begin{equation}\label{eq:equiv1k}
|w|_1 \leq \|w\|_{1,k} \leq (1+C_\mathrm{int}^2C_\mathrm{res}^2)^{1/2} \,|w|_{1}.
\end{equation}
Further, using Lemma~\ref{l:Nir}, \eqref{eq:proof1}, and \eqref{eq:stabPhicontr}, we have that 
\begin{equation}\label{eq:proof2}
\begin{aligned}
k^{2}C_nC_\eps \|\Phi\|_{L^6(D)}^2&\|w\|_{L^6(D)}\|w\|_0 \\
&\leq C_nC_\eps C^2_\mathrm{Nir}C_\mathrm{int}H k^{2d/3}\,\|\Phi\|_{1,k}^2 \|w\|_{L^6(D)}\,|w|_{1}\\
&\leq C_nC_\eps C^3_\mathrm{Nir}C_\mathrm{int}H k^{d-1}\, \|\Phi\|_{1,k,D}^2\, |w|_1^{d/3} k^{1-d/3}\|w\|_0^{1-d/3}\, |w|_1\\
&\leq C_nC_\eps C^3_\mathrm{Nir}C_\mathrm{int}^{2-d/3}H (k^{1-d/3}H^{1-d/3}) k^{d-1}\,\|\Phi\|_{1,k}^2\,|w|_1^2\\
&\leq  C_\mathrm{int}^{2-d/3}C_\mathrm{res}^{1-d/3} (C_{\mathrm{stab}, 0}(k))^{-1}\vartheta H|w|_1^2
\leq \frac{c_a}{2}|w|_1^2,
\end{aligned}
\end{equation}
where we applied assumption \eqref{eq:kerr} in the last step.
Therefore, with \eqref{eq:equiv1}, \eqref{eq:proof2}, and the inclusion $H^1(D)\subset L^6(D)$, we obtain for any $w \in \W$
\begin{equation*}
\begin{aligned}
\Re\, \calBl(\Phi;w,w) &\geq 
\Re\, \calBl(\Phi;w,w) =
(A\nabla w, \nabla w) - (k^2n(1+\eps|\Phi|^2)w,w)\\
&\geq c_A|w|^2_1 - C_\mathrm{int}^2C_\mathrm{res}^2C_n\,|w|_1^2 - \tfrac{c_A}{2}\,|w|_1^2 \\
& = \Big(\frac{c_A}{2} - C_\mathrm{int}^2C_\mathrm{res}^2C_n\Big)|w|_1^2.
\end{aligned}
\end{equation*}
The assertion follows with the definitions of $C_\mathrm{res}$ and $\gamma$.
\end{proof}

\subsection{Stability and error estimates for the auxiliary multiscale solution}

Based on $\Phi \in H^1(D)$ as above, we define the multiscale solution corresponding to \eqref{eq:NLHweakLin} as the solution $u_{\Phi,H} \in (1 - \calC_\Phi)\VH$ of
\begin{equation}\label{eq:LODaux}
\calBl(\Phi;u_{\Phi,H},\vH) = (f,\vH)
\end{equation}
for all $\vH \in (1 - \calC^*_\Phi)\VH$.

Based on Lemma~\ref{l:infsupW} and Lemma~\ref{l:infsup}, it is possible to show stability and error estimates for the solution $u_{\Phi,H}$ of~\eqref{eq:LODaux} as stated in the following lemma.

\begin{lemma}[Stability and approximation properties of the auxiliary multiscale solution]\label{l:stabapproxLOD}
Let $u_\Phi \in H^1(D)$ be the solution to~\eqref{eq:NLHweakLin} and $u_{\Phi,H} \in H^1(D)$ the solution to~\eqref{eq:LODaux}. If the assumptions of Lemma~\ref{l:infsupW} hold, we have that
\begin{equation}\label{eq:stabLODPhi}
\|u_{\Phi,H}\|_{1,k} \leq C_\mathrm{LOD}(k)\,\|f\|_0.
\end{equation}
with $C_\mathrm{LOD}(k):=\frac{C_\calB}{k\delta(k)\gamma}$ with $C_\calB$ from Proposition~\ref{prop:auxlinpb}, $\delta(k)$ from Lemma~\ref{l:infsup}, and $\gamma$ from Lemma~\ref{l:infsupW}.
Further, 
\begin{equation}\label{eq:errLODPhi}
\|u_{\Phi} - u_{\Phi,H}\|_{1,k} \leq C_\mathrm{err} \, H \,\|f\|_0, 
\end{equation}
where $C_\mathrm{err}:=\gamma^{-1}C_{\mathrm{int}}$.
\end{lemma} 

\begin{proof}
The proof is similar as for linear Helmholtz problems in \cite{Pet17,PetV20} and we present it here for convenience and to make the material self-consistent.

\emph{Proof of \eqref{eq:stabLODPhi}}: We prove the following inf-sup condition
\begin{equation*}
\adjustlimits\inf_{v_H\in V_H}\sup_{w_H\in V_H}\frac{\Re\,\calBl(\Phi;(1-\calC_\Phi)v_H, (1-\calC_\Phi^*)w_H)}{\|(1-\calC_\Phi)v_H\|_{1,k}\,\|(1-\calC_\Phi^*)w_H\|_{1,k}}\geq \frac{\delta(k)\gamma}{C_\calB}.
\end{equation*}
Recalling that $u_{\Phi, H}\in (1-\calC_\Phi)V_H$, the inf-sup condition implies the existence of $v_H\in V_H$ with $\|(1-\calC^*_\Phi)v_H\|_{1,k}=1$ such that
\begin{align*}
\|u_{\Phi, H}\|_{1,k}\leq \frac{C_\calB}{\delta(k)\gamma}\Re \calBl(\Phi;u_{\Phi, H}, (1-\calC^*_\Phi)v_H)=\frac{C_\calB}{\delta(k)\gamma}\Re(f,(1-\calC^*_\Phi)v_H)\leq \frac{C_\calB}{k\delta(k)\gamma}\|f\|_0,
\end{align*}
which yields \eqref{eq:stabLODPhi}.

To show the inf-sup condition, let $v_H\in V_H$ be arbitrary but fixed. Due to \eqref{eq:infsup}, there exists $w\in H^1(D)$ with $\|w\|_{1,k}=1$ such that
\[\Re\, \calBl(\Phi;(1-\calC_\Phi)v_H, w)\geq \delta(k)\|(1-\calC_\Phi)v_H\|_{1,k}.\]
We set $w_H=\IH w$ and observe that $(1-\calC^*_\Phi)w_H=(1-\calC^*_\Phi)w$. Therefore, with \eqref{eq:infsup} and the continuity of $\calBl$, we deduce that
\begin{align*}
\Re\, \calBl(\Phi;(1-\calC_\Phi)v_H, (1-\calC^*_\Phi)w_H)=\Re\,\calBl(\Phi;(1-\calC_\Phi)v_H, w)\geq \delta(k)\|(1-\calC_\Phi)v_H\|_{1,k}.
\end{align*} 
The inf-sup condition follows by the norm equivalence
\[\|(1-\calC^*_\Phi)w_H\|_{1,k}\leq C_\calB\gamma^{-1}\|w\|_{1,k}\]
due to the stability of the corrector problems; see, e.g., \cite{Pet17} for details.

\emph{Proof of \eqref{eq:errLODPhi}}: A simple calculation shows that $u_{\Phi,H}=(1-\calC_\Phi)\IH u_\Phi=(1-\calC_\Phi)u$ and hence, $u-u_{\Phi, H}=\calC_\Phi u\in \W$.
By the definition of $\calC_\Phi$ in \eqref{eq:corrPhi}, the error $u-u_{\Phi,H}$ therefore satisfies
\[\calBl(\Phi;\calC_\Phi u, \calC_\Phi u)=\calBl(\Phi; u, \calC_\Phi u)=(f, \calC_\Phi u).\]
Lemma \ref{l:infsupW} and \eqref{eq:IH1} then yield
\[\|u-u_{\Phi, H}\|_{1,k}^2\leq \gamma^{-1}(f, \calC_\Phi u)\leq \gamma^{-1}C_\mathrm{int} H \|f\|_0 \|u-u_{\Phi, H}\|_{1,k},\]
which finishes the proof.
\end{proof}

\subsection{Iterative multiscale approximation}
\label{ss:iterativemultiscale}

In this subsection, we define a sequence of multiscale solutions $\{\uH^m\}_{m \in \N}$. To this end, we abbreviate $\calC_{m} := \calC_{\scalebox{.65}{$\uH^{m}$}}$ and define \mbox{$\uH^m \in (1 - \calC_{m-1})\VH$} as the solution to
\begin{equation}\label{eq:NLHweakItLOD}
\calBl(\uH^{m-1};\uH^m,\vH) = (f,\vH) 
\end{equation}
for all $\vH \in (1-\calC^*_{m-1})\VH$ and given $\uH^0 \in H^1(D)$ as, e.g., a first-order approximation of $u^0$,
\begin{equation}\label{eq:LODinitErr}
\|\uH^0 - u^0\|_{1,k} \leq C_0\,H.
\end{equation}
Note that the functions $\{\uH^m\}_{m \in \N}$ fulfill similar stability properties as in Lemma~\ref{l:stabIt}. However, the stability constant~$C_\mathrm{stab}(k)$ needs to be replaced by~$C_\mathrm{LOD}(k)$ from Lemma~\ref{l:stabapproxLOD}.
We note that $C_{\mathrm{LOD}}(k)\approx C_\mathrm{stab}(k)$, i.e., those constants have the same scaling with respect to~$k$.

Based on these observations, we prove in the following theorem that the solutions $\{\uH^m\}_{m \in \N}$ are close to the iterative solutions $\{u^m\}_{m \in \N}$ of~\eqref{eq:NLHweakIt} with respect to the mesh size $H$.

\begin{theorem}\label{t:errItLOD} 
Define $\tilde C_\mathrm{stab}(k) := \max\{C_\mathrm{stab}(k),C_\mathrm{LOD}(k)\}$ with the constants $C_\mathrm{stab}(k)$ and $C_\mathrm{LOD}(k)$ from Proposition~\ref{prop:auxlinpb} and Lemma~\ref{l:stabapproxLOD}, respectively. Further, assume that 
\begin{equation}
2\,\tilde C^3_{\mathrm{stab}}(k)C_nC_\eps C^{3}_\mathrm{Nir}\,k^{d-1}\,\|f\|_0^2 \leq \vartheta. \label{eq:boundRHS}
\end{equation}
Then, we have that
\begin{equation}\label{eq:errit}
\|\uH^m - u^m\|_{1,k} \leq \vartheta^m\,\|\uH^0 - u^0\|_{1,k} + \frac{1}{1-\vartheta}C_\mathrm{err}\, H \,\|f\|_0
\end{equation}
for all $m \in \N$.
\end{theorem}

\begin{proof}
Let $m \in \N$ be fixed. 
We define an auxiliary solution $w^m \in H^1(D)$ as the solution to
\begin{equation}\label{eq:itAux}
\calBl(\uH^{m-1};w^m,v) = (f,v) 
\end{equation}
for all $v \in H^1(D)$. Note that due to \eqref{eq:boundRHS} and by Proposition~\ref{prop:auxlinpb} (as well as its adapted version for the series of multiscale solutions), we have the following stability estimates,  
\begin{align}
\|u^{m-1}_{H}\|_{1,k} &\leq C_\mathrm{LOD}(k)\,\|f\|_0,\label{eq:stab1}\\
\|u^m\|_{1,k} &\leq C_\mathrm{stab}(k)\,\|f\|_0,\label{eq:stab2}\\
\|w^m\|_{1,k} &\leq C_\mathrm{stab}(k)\,\|f\|_0.\label{eq:stab3}
\end{align}
Further, the error $e^m = u^m-w^m$ solves
\begin{equation*}
\begin{aligned}
\calBl(u^{m-1};e^m,v) &= \calBl(\uH^{m-1};w^m,v) - \calBl(u^{m-1};w^m,v)\\
&= (k^2\eps(|u^{m-1}|^2 - |\uH^{m-1}|^2) w^m,\vB)
\end{aligned}
\end{equation*}
for all $v \in H^1(D)$. 
Reusing ideas from Theorem~\ref{t:wellposedNL} and with the boundedness of $w$, $u^{m-1}$, and $\uH^{m-1}$ as quantified in \eqref{eq:stab1}--\eqref{eq:stab3}, we obtain
\begin{align*}
\|e^m\|_{1,k} &\leq C_\mathrm{stab}(k)\,\|k^2\eps(|u^{m-1}|^2 - |\uH^{m-1}|^2) w^m\|_{0}\\
&\leq  C_\mathrm{stab}(k)k^2C_nC_\eps\|w^m\|_{L^6(D)}\\*
&\qquad\cdot\big(\|u^{m-1}\|_{L^6(D)} + \|\uH^{m-1}\|_{L^6(D)}\big)\|\uH^{m-1} - u^{m-1}\|_{L^6(D)}\\
&\leq \vartheta\,\|\uH^{m-1} - u^{m-1}\|_{1,k}.
\end{align*}
Moreover, with Lemma~\ref{l:stabapproxLOD} we have that
\begin{equation*}
\|u^m_H - w^m\|_{1,k} \leq C_\mathrm{err} \, H \,\|f\|_0.
\end{equation*}
The previous two estimates and the triangle inequality result in
\begin{align*}
\|\uH^m - u^m\|_{1,k} &\leq \|e^m\|_{1,k} + \|u^m_H - w^m\|_{1,k}\\
&\leq \vartheta\,\|\uH^{m-1} - u^{m-1}\|_{1,k} + C_\mathrm{err} \, H \,\|f\|_0\\
&\leq \vartheta^m\,\|\uH^0 - u^0\|_{1,k} + C_\mathrm{err} \, H \,\|f\|_0\sum_{n \in \N_0}\vartheta^n\\
& \leq \vartheta^m\,\|\uH^0 - u^0\|_{1,k} + \frac{1}{1-\vartheta}C_\mathrm{err}\, H \,\|f\|_0. \qedhere
\end{align*}
\end{proof}

With Theorem~\ref{t:errItLOD}, we can directly derive the following result.

\begin{corollary}
Suppose that the assumptions of Theorem~\ref{t:errItLOD} are satisfied. Let $\uH^m$ be the iterative multiscale approximation in step $m$ starting from the initial value $\uH^0=0$.
Then, it holds that
\[\|u-\uH^m\|_{1,k}\leq \frac{1}{1-\vartheta}(C_{\mathrm{stab}, 0}(k)\vartheta^m+C_\mathrm{err} H)\|f\|_0.\]
That is, the choice $m \gtrsim |\log H|/|\log \vartheta|$ provides an approximation of $u$ with accuracy $\mathcal{O}(H)$. 
\end{corollary}
\begin{proof}
Approximate $u$ by the fixed-point iteration starting from the initial value $u^0=0$.
By the triangle inequality, we have $\|u-\uH^m\|_{1,k}\leq \|u-u^m\|_{1,k}+\|u^m-\uH^m\|_{1,k}$, where the latter term is estimated with Theorem~\ref{t:errItLOD}.
The a priori estimate for the fixed-point iteration together with \eqref{eq:stabLin0} then yields
\[\|u-u^m\|_{1,k}\leq \frac{\vartheta^m}{1-\vartheta}\|u^1-u^0\|_{1,k}\leq \frac{\vartheta^m}{1-\vartheta}C_{\mathrm{stab}, 0}(k)\|f\|_0.\qedhere\]
\end{proof}

We emphasize that the multiscale procedure introduced in this section (cf.~e.g.~\eqref{eq:LODaux}) is \emph{ideal} in the sense that for given $\Phi \in H^1(D)$, the correction operators $\calC_\Phi$ and $\calC^*_\Phi$ defined in \eqref{eq:corrPhi} and \eqref{eq:corrPhi*}, respectively, are global operations on the infinite-dimensional space $\W$. For practical computations, these operators are defined element-wise and truncated to local patches with $\ell$ layers of elements around each element. 
We will introduce and use this localization in the following subsection for our adaptive iterative multiscale approximation, which covers~\eqref{eq:NLHweakItLOD} as special case.

\begin{remark}\label{rem:finedisc}
Note that in practical computations, the correction operators also need to be discretized on a fine mesh with mesh size $h \ll H$ that resolves possible oscillations. This is required in order to obtain a fully practical method. We omit this last step in the present work and refer to~\cite{GalP15} for the corresponding analysis.
\end{remark}

\subsection{Adaptive iterative multiscale approximation}\label{ss:adaptive}
As mentioned, the goal of this subsection is two-fold. First, we localize the computation of the correction operators. Second, we introduce at the same time a strategy to (locally) decide where these correction operators need to be updated from one iteration step to the other.
The motivation for this adaptive strategy is that updating the correction operators in each step is computationally rather expensive and might be unnecessary if the multiscale solutions only change in certain parts of the domain.
Inspired by \cite{HelM19,HelKM20}, we will introduce an error indicator for the corrector if the function $\Phi\in H^1(D)$ is ``perturbed''. 
The numerical analysis of the resulting adaptive iterative multiscale algorithm is based on Theorem~\ref{t:errItLOD} from the previous section and the observation that the localization as well as the updating strategy both are small perturbations thereof.

We first introduce the localization of the correction operators.
Recall the definition of $\Nb(T)$ from Section \ref{s:multiscale}. We inductively define the $\ell$-layer patch for $\ell\in \N$ via 
\begin{equation*}
\Nb^\ell(T)=\Nb(\Nb^{\ell-1}(T)),\quad \ell \geq 2,\quad \text{ and }\quad\Nb^1(T)=\Nb(T).
\end{equation*}
The kernel space $\W$ is restricted to such element patches via
\begin{equation*} 
\W(\Nb^\ell(T)):=\{w\in \W \,: \,w\vert_{D\setminus \Nb^\ell(T)}=0 \}.
\end{equation*}
Given a function $\Phi\in H^1(D)$ and an element $T\in \calT_H$, we define the $\ell$-layer element corrector $\calC_{\Phi, T}^\ell$ via the solution to a truncated and element-based corrector problem as follows. For any $v_H\in V_H$, we seek $\calC_{\Phi, T}^\ell v_H\in \W(\Nb^\ell(T))$ that solves
\begin{equation}\label{eq:elementcorrector}
\calB_{\mathrm{lin}, \Nb^\ell(T)}(\Phi; \calC_{\Phi, T}^\ell v_H, w)=\calB_{\mathrm{lin}, T}(\Phi; v_H, w) 
\end{equation}
for all $w\in \W(\Nb^{\ell}(T))$.
Here, $\calB_{\mathrm{lin}, S}$ denotes the restriction of $\calBl$ to the subdomain $S\subset D$.
The globally defined $\ell$-layer corrector $\calC_\Phi^\ell$ is then given as the sum of the element correctors, i.e., 
\begin{equation*} 
\calC_\Phi^\ell:=\sum_{T\in \calT_H}\calC_{\Phi, T}^\ell.
\end{equation*}
The adjoint correction operator $\calC_\Phi^{\ell,*}$ is defined analogously.
If $\Phi\in H^1(D)$ satisfies \eqref{eq:stabPhicontr} and the conditions \eqref{eq:res}--\eqref{eq:kerr} are fulfilled, the coercivity \eqref{eq:infsupW} from Lemma~\ref{l:infsupW} implies that \eqref{eq:elementcorrector} is well-posed.
Note that in order to compute~$\calC_{\Phi, T}^\ell$ it suffices to know $\Phi$ on $\Nb^\ell(T)$. 
The error between $\calC_\Phi^\ell$ and $\calC_\Phi$ is decaying exponentially in $\ell$, which directly carries over from the linear case \cite{GalP15,BroGP17}.
\begin{lemma}\label{l:trunccorrectorerr}
Let $\Phi\in H^1(D)$ satisfy \eqref{eq:stabPhicontr}.
There exist $0<\beta<1$ and $C_{\mathrm{loc}}>0$, independent of $H$, $\ell$, and $k$, such that for any $v_H\in V_H$
\begin{equation*}
\|(\calC_\Phi-\calC_\Phi^\ell)v_H\|_{1,k}\leq C_{\mathrm{loc}}\ell^{d/2}\beta^\ell\|v_H\|_{1,k}.
\end{equation*}
\end{lemma}
In the localized variant of the iterative multiscale approximation from Section~\ref{ss:iterativemultiscale}, $\calC_{\scalebox{.65}{$\uH^m$}}$ is replaced by $\calC_{\scalebox{.65}{$\uH^m$}}^\ell$.
As discussed, in this approach all localized element correctors $\calC_{\scalebox{.65}{$\uH^{m}$}, T}^\ell$ are newly computed in each step even if $\uH^m$ has only slightly changed on the patch $\Nb^\ell(T)$.
To introduce our adaptive approach, we need to estimate the error between element correction operators $\calC_{\Phi, T}^\ell-\calC_{\Psi, T}^\ell$ for two different functions $\Phi, \Psi\in H^1(D)$.
In this setting, one should consider $\calC_{\Psi, T}^\ell$ as being \emph{available}, i.e., it has already been computed for instance in a previous iteration step, whereas $\calC_{\Phi, T}^\ell$ is not available (only the function $\Phi$). 
In our implementation, we use the following error indicator
\begin{equation}\label{eq:errIndT}
E(\calC_{\Phi, T}^\ell, \calC_{\Psi, T}^\ell)^2:=\sum_{K\in \Nb^\ell(T)}\|n\varepsilon(|\Phi|^2-|\Psi|^2)\|_{L^\infty(K)}^2\max_{v\vert_T\,:\, v\in V_H}\frac{\|\chi_T v-\calC_{\Psi, T}^\ell v\|_{0, K}^2}{\|v\|_{0,T}^2},
\end{equation}
where $\chi_T$ denotes the indicator function for the element $T$. 
Note that the error indicator avoids the computation of $\calC_{\Phi, T}^\ell$.
We now have the following result.
\begin{lemma}\label{l:erroindic}
Let $\Phi, \Psi\in H^1(D)$ satisfy \eqref{eq:stabPhicontr}.
Then, for any $\vH\in \VH$ it holds
\begin{equation}\label{eq:indicloc}
\|(\calC_{\Phi, T}^\ell-\calC_{\Psi, T}^\ell)v_H\|_{1,k}\leq \gamma^{-1}E(\calC_{\Phi, T}^\ell, \calC_{\Psi, T}^\ell)\, \|v_H\|_{1,k,T}
\end{equation}
and there exists $C_{\mathrm{ol}}>0$, independent of $H$, $\ell$, and $k$, such that for any $\vH\in \VH$ it holds
\begin{equation}\label{eq:indicglobal}
\|(\calC_{\Phi}^\ell-\calC_{\Psi}^\ell)v_H\|_{1,k}\leq C_{\mathrm{ol}} \ell^{d/2}\gamma^{-1}\big(\max_{T\in \calT_H}E(\calC_{\Phi, T}^\ell, \calC_{\Psi, T}^\ell)\big)\,\|v_H\|_{1,k}.
\end{equation}
\end{lemma}
The proof of this lemma and a detailed discussion of the error indicator are postponed to the appendix in order to ease the reading.
We emphasize that the main result in Theorem~\ref{t:errAdItLOD} does not depend on the exact form of the error indicator, but only needs the estimates of Lemma~\ref{l:erroindic}.

Generalizing the setting of the previous section, we now define a sequence of adaptive multiscale solutions $\{\uHT^m\}_{m \in \N}$ as outlined in Algorithm~\ref{alg:adaptive} below.\\

\begin{breakablealgorithm}
	\caption{Adaptive iterative multiscale approximation}
	\label{alg:adaptive}
	\begin{flushleft}
		\textbf{input:} tolerance $\texttt{tol}$, starting value $\uHT^0\in V_H$, mesh $\calT_H$, oversampling parameter $\ell$ 
	\end{flushleft}
	\begin{algorithmic}[1] 
		\For {$m=1,2,\ldots$} 
		\ForAll {$T \in \calT_H$}
		\If {$m = 1$}	
			\State $E_{T}^m \gets \infty$
		\Else 
			\State $E_{T}^m \gets E(\calC_{\scalebox{.65}{$\uHT^{m-1}$}, T}^\ell, \tilde{\calC}_{m-2, T}^\ell)$ \Comment{does \textbf{not} explicitly require $\calC_{\scalebox{.65}{$\uHT^{m-1}$}, T}^\ell$}
		\EndIf
		\EndFor
		\State $\mathcal{M}_m \gets \{T\in\calT_H\,:\,E_T^m>\texttt{tol}\}$ \Comment{elements for which the corrector is updated}
		\ForAll {$T \in \calT_H$}
			\If {$T \in \mathcal{M}_m$}
				\State compute $\calC_{\scalebox{.65}{$\uHT^{m-1}$},T}^\ell$ and $\calC_{\scalebox{.65}{$\uHT^{m-1}$},T}^{\ell,*}$
				\State  $\tilde{\calC}_{m-1, T}^\ell \gets \calC_{\scalebox{.65}{$\uHT^{m-1}$}, T}^\ell$ \Comment{update corrector}
			\Else 
				 \State $\tilde{\calC}_{m-1, T}^\ell \gets \calC_{m-2, T}^\ell$ \Comment{reuse the old corrector}
			\EndIf
		\State $\tilde{\calC}_{m-1}^\ell \gets \sum_{T\in \calT_H}\tilde{\calC}_{m-1, T}^\ell$
		\State $\tilde{\calC}_{m-1}^{\ell,*} \gets \sum_{T\in \calT_H}\tilde{\calC}_{m-1, T}^{\ell,*}$
		\EndFor
		
		\State compute $\uHT^{m}\in (1-\tilde{\calC}_{m-1}^\ell)V_H$ as the solution to
		\begin{equation}\label{eq:itmssol} 
		\calBl(\uHT^{m-1}; \uHT^{m}, v_H)=(f, \overline{v_H})
		\end{equation}
		\quad\, for all $v_H\in (1-\tilde{\calC}_{m-1}^{\ell, *})V_H$
		
		\EndFor
	\end{algorithmic}
	\begin{flushleft}
		\textbf{output:} sequence $\{\uHT^m\}_{m\in \N}$
	\end{flushleft}	
\end{breakablealgorithm} \phantom{.}\\[1.5ex]
Let us shortly explain the adaptive algorithm. For $m=1$, all correctors are computed based on the starting value $\uHT^0$ and the LOD solution is computed as described in the previous section. In particular, if we choose the same starting value, we have $\uHT^1=\uH^1$ (up to the localization of the correctors).
In the subsequent iterations, we decide, for each element $T$, whether to compute a new corrector based upon our error indicator $E$. Otherwise, the available corrector from the previous iteration(s) is used.
Then, we assemble the LOD stiffness matrix using the mixture of newly computed and reused correctors and the previous LOD iterate $\uHT^{m-1}$ for the nonlinearity. 
We emphasize that we do not reuse contributions to the stiffness matrix as suggested for lagging or perturbed linear diffusion problems in \cite{HelM19,HelKM20}.
In practice, the loop of the above algorithm will of course be terminated using the residual 
$\texttt{res}^m$ in the $m$th step
as stopping criterion (together with a second tolerance \texttt{TOL} as input). 
Let us also emphasize that the loop over all elements (lines $10$ to $19$) allows for parallel computations. 
Finally, we note that in the extreme cases $\texttt{tol}=0$ or $\texttt{tol}=\infty$, we obtain the algorithm from Section \ref{ss:iterativemultiscale} or a fixed-point iteration with fixed multiscale space $(1-\calC_{\scalebox{.65}{$\uHT^0$}}^\ell)V_H$, respectively.
\begin{remark}[Practical choice of \texttt{tol}]\label{rem:tol}
While Theorem~\ref{t:errAdItLOD} below indicates how to choose \texttt{tol} to obtain a certain (guaranteed) accuracy, such a choice may be pessimistic and lead to an overly large number of corrector updates.
In practice, we suggest the following choice of the tolerance that is also used in our numerical experiments: fix a tolerance factor $\zeta_{\mathrm{tol}}\in [0,1]$ and set the tolerance in the $m$th iteration step to \[\texttt{tol}:=\zeta_{\mathrm{tol}}\, \bigl(\max_{T\in \calT_H} E_T^m\bigr).\]
With this strategy, \texttt{tol} changes in every iteration, but the a priori estimate of Theorem~\ref{t:errAdItLOD} below can still be applied by considering the maximal tolerance over all iterations.
Typically, the error indicators and thus the practically chosen tolerance get smaller over the iterations in our experiments.
\end{remark}
The well-posedness of~\eqref{eq:itmssol} does not follow from Lemma~\ref{l:stabapproxLOD} because not all element correctors are newly computed.
However, the well-posedness of the algorithm as well as the a priori estimate for the error $\uHT^m-u^m$  can be shown in a similar fashion as in Lemma~\ref{l:stabapproxLOD} and Theorem~\ref{t:errItLOD}.
The additional errors by the localization and the adaptive update of the correctors are only small perturbations for sufficiently large $\ell$ and sufficiently small $\texttt{tol}$,  cf.~Lemma~\ref{l:trunccorrectorerr} and Lemma~\ref{l:erroindic}.
More precisely, we have the following analog of Theorem~\ref{t:errItLOD}.
The proof with all technical details is again postponed to the appendix.
\begin{theorem}\label{t:errAdItLOD}
Assume that $\ell$ and $\mathtt{tol}$ satisfy
\begin{equation}\label{eq:oversampl}
\ell\gtrsim\Big|\log\Bigl(\frac{\delta(k)\gamma^2}{36\,C_\calB^3C_{\mathrm{int}}^3C_{\mathrm{loc}}}\Bigr)\Big| / |\log(\beta)| \quad \text{and}\quad \ell\gtrsim\Big|\log\Bigl(\frac{C_\calB}{\gamma C_{\mathrm{loc}}}\Bigr)\Big| / |\log(\beta)|
\end{equation}
as well as
\begin{equation}\label{eq:condtol}
\mathtt{tol}\leq \frac{\delta(k)\gamma^3}{36\,C_{\mathrm{int}}^3C_\calB^3C_{\mathrm{ol}}\ell^{d/2}}\quad \text{and}\quad  \mathtt{tol}\leq \frac{C_\calB}{C_{\mathrm{ol}}\ell^{d/2}}. 
\end{equation}
Define $\hat{C}_\mathrm{stab}(k):=\max\{C_\mathrm{stab}, \frac{18\,C_\mathrm{int}^2 C_\calB^2}{k\delta(k)\gamma^2}\}$ with the constant $C_\mathrm{stab}(k)$ from Proposition~\ref{prop:auxlinpb}.
Further, suppose that
\begin{equation}\label{eq:boundRHSadap}
2\hat{C}_\mathrm{stab}^3(k)C_nC_\varepsilon C_\mathrm{Nir}^3 k^{d-1}\|f\|_0^2\leq \vartheta.
\end{equation}
Then the sequence $\{\uHT^m\}_{m \in \N}$ of adaptive iterative multiscale approximations is well-defined and satisfies the error estimate
\begin{equation}\label{eq:erroradap}
\begin{aligned}
\|\uHT^m-u^m\|_{1,k} &\leq \vartheta^m\|\uHT^0-u^0\|_{1,k}\\
&\quad+\frac{2}{1-\vartheta}\gamma^{-1}C_\mathrm{int}\big(H+C_\calB C_{\mathrm{loc}}\,\ell^{d/2}\beta^\ell+C_\calB C_\mathrm{stab}(k)C_{\mathrm{ol}}\,\ell^{d/2}\,\gamma^{-1}\,\mathtt{tol} \big)\|f\|_0.
\end{aligned}
\end{equation}
\end{theorem} 

We emphasize that the first condition of \eqref{eq:oversampl} and \eqref{eq:condtol}, respectively, is the dominant condition.
Condition~\eqref{eq:oversampl} requires~$\ell$ to grow with increasing wave numbers~$k$.
Under the assumption of polynomial stability (cf.~the discussion after Assumption~\ref{a:stabLin0}), $\ell$ has to grow logarithmically in $k$ as for the linear Helmholtz equation, see \cite{GalP15,Pet17,PetV20}.
The tolerance~\texttt{tol} has to decrease with growing $k$, in particular under the assumption of polynomial stability as above, we have $\mathtt{tol}\lesssim k^{-(q+1)}$.
We emphasize that this qualitative behavior is expected because for growing wave numbers, the nonlinearity becomes more dominant which has to be compensated by more updates of the correctors.
The error between $\uHT^m$ and $u^m$ in the energy norm in \eqref{eq:erroradap} is essentially of order $H+\beta^\ell+\mathtt{tol}$. The order $H+\beta^\ell$ occurs also in the study of the linear Helmholtz equation \cite{GalP15,Pet17,PetV20} and suggests to choose $\ell\approx |\log(H)|$ to obtain a linear rate.
The additional term~\texttt{tol} obviously comes from the reuse of correctors and can be made arbitrarily small at the price of growing computational costs.

In other words, as mentioned, for $\mathtt{tol}=0$ we obtain the localized version of the algorithm in Section~\ref{ss:iterativemultiscale}.
All additional terms and factors in Theorem \ref{t:errAdItLOD} in comparison to Theorem \ref{t:errItLOD} are caused by the handling of the additional localization error.
We note that the factor $18$ in $\hat C_{\mathrm{stab}}(k)$ and the factor $3$ in \eqref{eq:erroradap} can be made smaller for $\mathtt{tol}=0$ by a close inspection of the proof.
In fact for zero tolerance, all steps where we switch between $\tilde{\calC}_{m-1}^\ell$ and $\calC_{\scalebox{.65}{$\uHT^{m-1}$}}^\ell$ can be omitted.

\section{Numerical Examples}\label{s:numex}
In this section, we present numerical examples to investigate the practical performance of our proposed iterative multiscale method. Note that in our studies the exact solution $u$ of the nonlinear problem \eqref{eq:NLHweak} is not explicitly known. Therefore, we compare the iterative multiscale solutions computed in \eqref{eq:itmssol} with a reference solution $u_h$, which is computed as a standard finite element approximation of \eqref{eq:NLHweak}. We implicitly assume that $h$ is small enough such that $u_h$ is a reasonable approximation of $u$. In particular, the mesh size $h$ needs to resolve the multiscale features in the coefficients $A$, $n$, and $\eps$. Further, we measure all errors in the energy norm $\|\cdot\|_{1,k}$.

We consider the domain $D = (0,1)^2$ and choose the fine mesh size $h = 2^{-9}$ and let the coefficients vary on a mesh $\calT_\eta$ on the scale $\eta = 2^{-7}$. All numerical examples are computed with Python using an adapted version of the software \texttt{gridlod}~\cite{HelK19}, which is based on the Petrov-Galerkin version of the LOD method described above. The code is available at \url{https://github.com/BarbaraV/gridlod-nonlinear-helmholtz}. Note that for the first two experiments, we use the update strategy as described in Remark~\ref{rem:tol} which sets the tolerance in each step. To investigate the dependence of the error on \texttt{tol}, we present a third example where we use various (fixed) tolerances as considered in Theorem~\ref{t:errAdItLOD}.

We emphasize that the smallness assumptions (cf.~ 
\eqref{eq:stabRHS2} and~\eqref{eq:boundRHS}) are not fulfilled in our examples -- provided that the stability constants are not significantly smaller than $1$ -- due to relatively large norms $\|f\|_0$. This, however, seems to indicate that the smallness assumptions are not sharp and the method is reliable beyond the theoretical setting. Nevertheless, a relaxed version of the condition still seems to hold as observed by numerical investigations. 

\subsection{Example 1: point source}

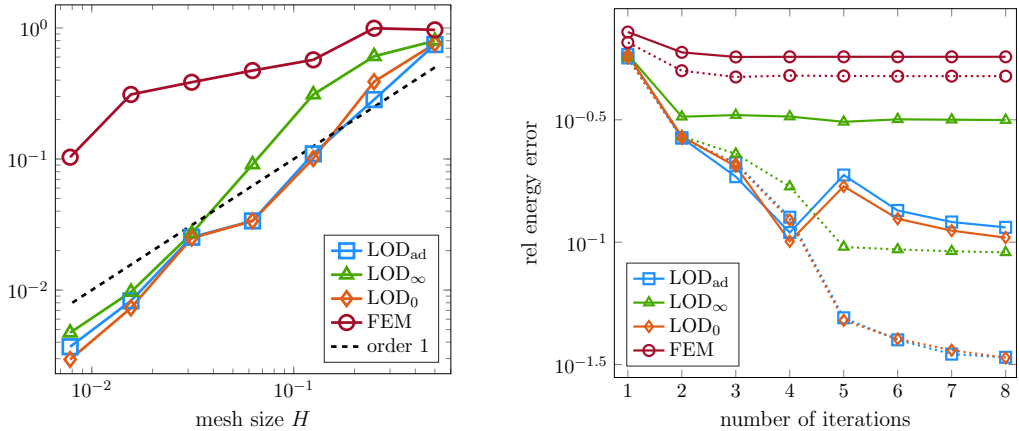
\begin{figure}
	\begin{center}
		\scalebox{0.69}{
			\begin{tikzpicture}
			
			\begin{axis}[%
			width=3.0in,
			height=2.8in,
			scale only axis,
			at={(0.758in,0.481in)},
			scale only axis,
			xmin=0.0066,
			xlabel={\large mesh size $H$},
			xmax=0.6,
			xminorticks=true,
			ymode=log,
			xmode=log,
			ymin=2.3e-03,
			ymax=1.5,
			yminorticks=true,
			axis background/.style={fill=white},
			legend style={legend cell align=left, align=left,
				draw=white!15!black},
			legend pos = south east,
			]
				
			\addplot [color=myBlue, mark=square, line width=1.5pt, mark size=4.0pt]
			table[row sep=crcr]{%
				0.5	0.746799302418362\\
				0.25	0.28410593713032\\
				0.125	0.10949799076263\\
				0.0625	0.0337103064511593\\
				0.03125	0.0252868588257529\\
				0.015625	0.00827572047991259\\
				0.0078125	0.00369509499152635\\
			};
			\addlegendentry{$\mathrm{LOD}_\mathrm{ad}$}
			
			\addplot [color=myGreen, mark=triangle, line width=1.5pt, mark size=4.0pt]
			table[row sep=crcr]{%
				0.5	0.79791022771592\\
				0.25	0.607784694636488\\
				0.125	0.310414076142366\\
				0.0625	0.0908959396386922\\
				0.03125	0.027063418241726\\
				0.015625	0.00967306087371738\\
				0.0078125	0.00472897666793636\\
			};
			\addlegendentry{$\mathrm{LOD}_\infty$}
			
			\addplot [color=myOrange, mark=diamond, line width=1.5pt, mark size=4.0pt]
			table[row sep=crcr]{%
				0.5	0.757654815562343\\
				0.25	0.389238415642447\\
				0.125	0.100770383091331\\
				0.0625	0.0336975627146933\\
				0.03125	0.0249281996361832\\
				0.015625	0.00730187628160989\\
				0.0078125	0.00295418680649401\\
			};
			\addlegendentry{$\mathrm{LOD}_0$}
			
			\addplot [color=myRed, mark=o, line width=1.5pt, mark size=4.0pt]
			table[row sep=crcr]{%
				0.5	0.967389512339179\\
				0.25	0.995268972849217\\
				0.125	0.571324763621321\\
				0.0625	0.473139658805676\\
				0.03125	0.385334598336857\\
				0.015625	0.31082177443248\\
				0.0078125	0.103207625204378\\
			};
			\addlegendentry{FEM}
		
			\addplot [color=black, dashed,line width=1.5pt]
			table[row sep=crcr]{%
				0.5		0.5\\
				0.0078125	0.0078125\\
			};
			\addlegendentry{order $1$}

			\end{axis}
			\end{tikzpicture}%
		}
		\hspace{.5cm}
		\scalebox{0.69}{
			\begin{tikzpicture}
			
			\begin{axis}[%
			width=3.05in,
			height=2.8in,
			scale only axis,
			at={(0.772in,0.481in)},
			scale only axis,
			xmin=0.75,
			xlabel={\large number of iterations},
			xmax=8.2,
			ymode=log,
			ymin=2.8e-2,
			ymax=0.9,
			ylabel={\large rel energy error},
			yminorticks=true,
			axis background/.style={fill=white},
			legend style={legend cell align=left, align=left, draw=white!15!black},
			legend pos = south west
			]
				
			\addplot [color=myBlue, mark=square, line width=1.2pt, mark size=3.0pt]
			table[row sep=crcr]{%
				1	0.583985712618789\\
				2	0.266778879066947\\
				3	0.185250943066988\\
				4	0.10949799076263\\
				5	0.188140498991427\\
				6	0.134758722946591\\
				7	0.120935204188977\\
				8	0.114964634364349\\
			};
			\addlegendentry{$\mathrm{LOD}_\mathrm{ad}$}
			
			\addplot [color=myBlue, mark=square, dotted, line width=1.2pt, mark size=3.0pt, mark options={solid}, forget plot]
			table[row sep=crcr]{%
				1	0.566255485046909\\
				2	0.266640737247647\\
				3	0.210875832139256\\
				4	0.126506106161603\\
				5	0.0490863609855369\\
				6	0.0399154490486818\\
				7	0.03490026710594\\
				8	0.0338610477939377\\
			};
			
			\addplot [color=myGreen, mark=triangle, line width=1.2pt, mark size=3.0pt, mark options={solid}]
			table[row sep=crcr]{%
				1	0.583985712618789\\
				2	0.325652761123258\\
				3	0.330788981359923\\
				4	0.326180408721868\\
				5	0.310414076142366\\
				6	0.317679890851589\\
				7	0.316637419859261\\
				8	0.315573581846614\\
			};
			\addlegendentry{$\mathrm{LOD}_\infty$}
			
			\addplot [color=myGreen, mark=triangle, dotted, line width=1.2pt, mark size=3.0pt, mark options={solid}, forget plot]
			table[row sep=crcr]{%
				1	0.566255485046909\\
				2	0.270691803704024\\
				3	0.229595730109755\\
				4	0.168728737149417\\
				5	0.0956356845373586\\
				6	0.0933747942566275\\
				7	0.091838433424397\\
				8	0.0908959396386922\\
			};
			
			\addplot [color=myOrange, mark=diamond, line width=1.2pt, mark size=3.0pt, mark options={solid}]
			table[row sep=crcr]{%
				1	0.583985712618789\\
				2	0.27049943758259\\
				3	0.204205518388288\\
				4	0.100770383091331\\
				5	0.169117816340225\\
				6	0.12465385976353\\
				7	0.11148297000437\\
				8	0.104424416867215\\
			};
			\addlegendentry{$\mathrm{LOD}_0$}
			
			\addplot [color=myOrange, mark=diamond, dotted, line width=1.2pt, mark size=3.0pt, mark options={solid}, forget plot]
			table[row sep=crcr]{%
				1	0.566255485046909\\
				2	0.265388691840433\\
				3	0.209234990895319\\
				4	0.12338603420678\\
				5	0.047949603482452\\
				6	0.0402486624938315\\
				7	0.0362191467231747\\
				8	0.0336975627146933\\
			};
			
			\addplot [color=myRed, mark=o, line width=1.2pt, mark size=3.0pt]
			table[row sep=crcr]{%
				1	0.722517464841298\\
				2	0.596699969453035\\
				3	0.571324763621321\\
				4	0.57217600676365\\
				5	0.57230242759741\\
				6	0.572262528306353\\
				7	0.572262143101785\\
				8	0.572262628765595\\
			};
			\addlegendentry{FEM}
			
			\addplot [color=myRed, mark=o, dotted, line width=1.2pt, mark size=3.0pt, mark options={solid},forget plot]
			table[row sep=crcr]{%
				1	0.65521816686242\\
				2	0.5019247651342\\
				3	0.473139658805676\\
				4	0.479671456344483\\
				5	0.477709164980791\\
				6	0.476757234327912\\
				7	0.477007831119992\\
				8	0.477015981743695\\
			};
				
			\end{axis}
			\end{tikzpicture}%
		}
	\end{center}
	\caption[]
	{\small Relative energy errors of different iterative methods (left) and errors with respect to the number of iterations (right) for Example~$1$. {The solid and dotted lines in the right plot are obtained for $H = 2^{-3}$ and $H=2^{-4}$, respectively.}} \label{fig:exp1Err}
\end{figure}

For the first example, we consider the right-hand side 
\begin{equation*}
f(x) = \left\{
\begin{aligned}
&10000\,\exp\bigg(-\frac{1}{1-\big(\tfrac{|x-x_0|}{0.05}\big)}\bigg) && \text{ if }\tfrac{|x-x_0|}{0.05} < 1,\\
&0 && \text{ else,}
\end{aligned}\right.
\end{equation*}
with center $x_0 = (0.5,0.5)$. 
Further, we have $k = 17$ and the domain where the nonlinearity is active is given by $[0.55,0.75]\times[0.25,0.45]$. 
The scalar coefficients $A$, $n$, and $\eps$ are chosen as piecewise constant coefficients on the finite element mesh $\calT_\eta$. More precisely, the values on each element are obtained as independent and uniformly distributed random numbers within the intervals $[0.5,3]$, $[0.5,1]$, and $[0,9.4]$ for $A$, $n$, and~$\eps$, respectively. Note, however, that the values of $A$ and $n$ in $D\setminus[0.15,0.85]^2$ are explicitly set to $1$ in order to satisfy the last condition of Assumption~\ref{a:coeff}.

We compute the reference solution on the mesh $\calT_h$, as mentioned above, which is obtained by an iterative solution similarly as described in~\eqref{eq:NLHweakIt}. Note that this discrete iteration is well-posed as pointed out in Remark~\ref{rem:fineFEM}. We iterate until the (relative) residual reaches the threshold of $10^{-12}$, which requires $45$ iterations. 

In Figure~\ref{fig:exp1Err}~(left), we show the errors of adaptively updated LOD solutions with $\ell = 2$ for different values of $H$ (LOD$_{\mathrm{ad}}$, {\protect \tikz{ \draw[line width=1.pt, myBlue] (0,0) rectangle (0.18,0.18);}}). More precisely, for any $H$, we show the smallest error that is obtained within the first $20$ steps of the iterative multiscale method. For this particular example, we use the update strategy described in Remark~\ref{rem:tol} with $\zeta_\mathrm{tol} = 0.5$. 
That is, in each step we only update element correctors whose error indicator (as defined in~\eqref{eq:errIndT}) is larger than half the value of the maximal error indicator in this iteration. Note that, in the first iteration, all element correctors have to be computed. Afterwards, the maximal update percentages within the first $20$ iterations for $H = 2^{-1},\,2^{-2},\,2^{-3},\,2^{-4},\,2^{-5},\,2^{-6},\,2^{-7}$ are 
$50,\,6.25,\,6.25,\,3.52,\,1.56,\,0.88,\,0.37$ 
(in \%), respectively. We emphasize that these percentages are with respect to the total number of elements in the corresponding mesh, which of course increases when $H$ is decreased.
In this example, the iterative method reaches a fixed point for all choices of $H$ within the first $15$ iterations. Further, the choice $\ell = 2$ shows to be sufficient, which is in line with the practical choices of $\ell$ in the linear case; see, e.g.~\cite{BroGP17,PetV20}. Overall, the error shows a linear convergence rate underlining the findings of Theorem~\ref{t:errAdItLOD}.

\begin{figure*}
	\centering
	\captionsetup[subfigure]{labelformat=empty}
	\begin{subfigure}[b]{0.495\textwidth}
		\centering 
		\includegraphics[width=\textwidth]{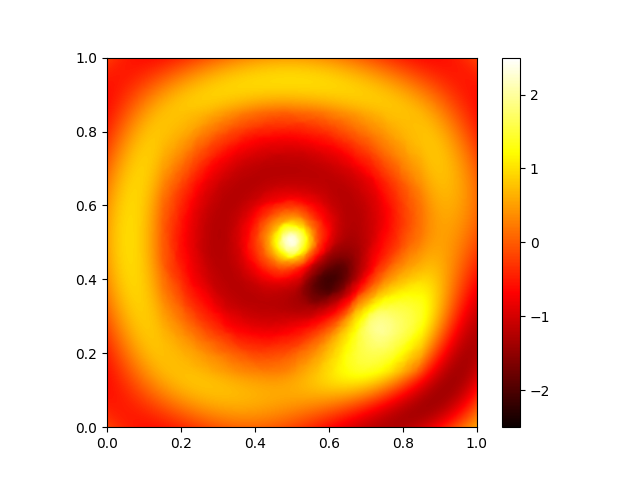}
		\caption[]%
		{\small reference solution}   
		\label{fig:exp1exact} 		
	\end{subfigure}
	\begin{subfigure}[b]{0.495\textwidth}  
		\centering
		\includegraphics[width=\textwidth]{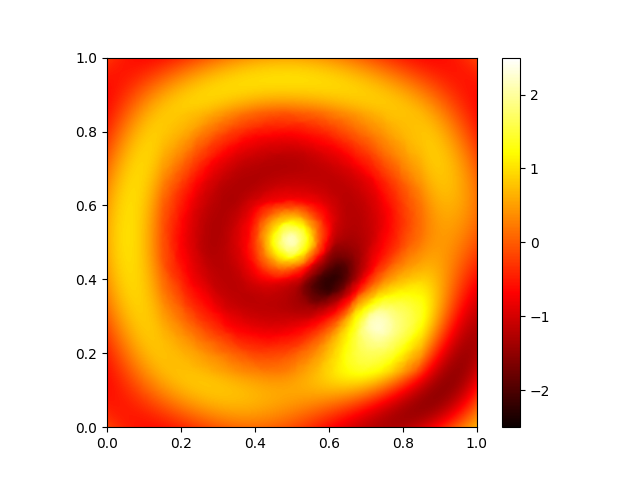}
		\caption
		{\small $\mathrm{LOD}_\mathrm{ad}$}
		\label{fig:exp1LOD}
	\end{subfigure}
	\vskip\baselineskip
	\begin{subfigure}[b]{0.495\textwidth}   
		\centering 
		\includegraphics[width=\textwidth]{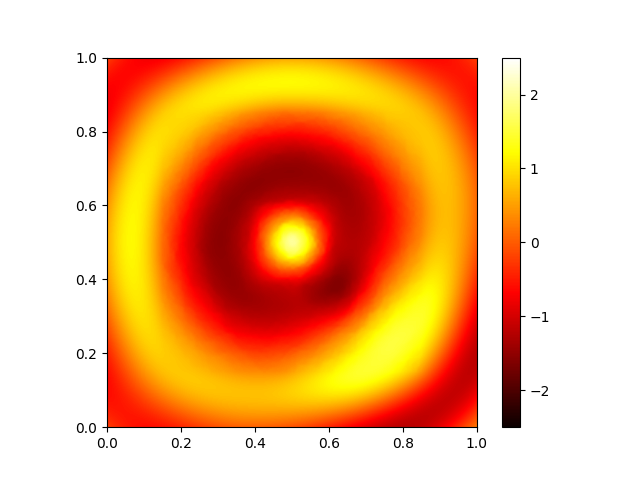}
		\caption[]%
		{\small $\mathrm{LOD}_\infty$}    
		\label{fig:exp1LODinf}
	\end{subfigure}
	\begin{subfigure}[b]{0.495\textwidth}   
		\centering 
		\includegraphics[width=\textwidth]{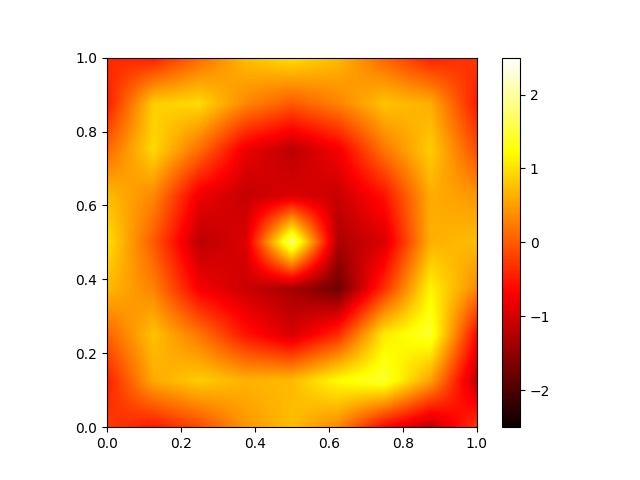}
		\caption[]%
		{\small FEM}    
		\label{fig:exp1FEM}
	\end{subfigure}
	\caption[]
	{\small Real part of different final approximations on the scale $H=2^{-3}$ and the corresponding reference solution for Example~$1$.} 
	\label{fig:exp1Approx}
\end{figure*}

As comparison, we also present in Figure~\ref{fig:exp1Err}~(left) classical finite element approximations on the same scales~$H$~(FEM, {\protect \tikz{ \draw[line width=1.pt, myRed] circle (0.6ex);}}) as well as iterative LOD solutions for which the correctors are only computed in the first iteration and never updated afterwards (LOD$_\infty$, {\protect \tikz{ \draw[line width=1.pt, myGreen] (0,0) --(0.18,0) -- (0.09,0.18) -- (0,0);}}) and LOD solutions where all correctors are always updated (LOD$_0$, {\protect \tikz{ \draw[line width=1.pt, myOrange] (0.07,0) --(0.14,0.1) -- (0.07,0.2) -- (0.,0.1) -- (0.07,0);}}).
The finite element approximation shows a suboptimal convergence rate and the error only improves when $H$ is close to the scale $\eta$, i.e., where the coefficients are actually resolved by the coarse mesh. Then again, the LOD approximation which never updates the correctors starts off with a slow convergence similarly to the finite element approximation but heavily improves for smaller values of $H$, where the error is close to the one where the correctors are {(partially)} updated in every step. This behavior can be explained by the fact that the nonlinearity is only active in a small portion of the domain. Further, if $H$ approaches $\eta$, the correctors become less influential. That is, for smaller values of $H$, very few updates of the correctors are already sufficient, which is also indicated by the small update percentages mentioned above. Finally, it is important to note that our partial update strategy leads to errors very close to the approximations where all corrections are updated in every step ($\mathtt{tol}=0$), but with significantly less recomputations.

The influence of the updates is especially important for larger values of $H$, see Figure~\ref{fig:exp1Err}~(right). There, we show the development of the relative errors of the different approximation methods for the particular cases $H=2^{-3}$ (solid lines) and $H = 2^{-4}$ (dashed lines). The plot shows that the error significantly improves with the number of iterations provided that the correctors are (partially) updated. If the correctors are not updated after the first iteration, a faster stagnation can be observed. We note that the errors do not necessarily decrease monotonically, but this does not contradict our above theory.

Finally, we present in Figure~\ref{fig:exp1Approx} the final solutions for three of the different approaches discussed above ($\mathrm{LOD}_\mathrm{ad}$, $\mathrm{LOD}_\infty$, and FEM) on the scale $H=2^{-3}$ after $20$ iterations as well as the reference solution. 
The figure shows that the LOD approximation with corrector updates leads to a very good approximation already on the relatively coarse scale $H=2^{-3}$, while not updating the correctors after the first step deteriorates the behavior of the solution in the nonlinear domain. As expected, the finite element approximation, which does not take into account any variations of the coefficients, produces the worst result.

\subsection{Example 2: beam} 
\begin{figure*}
	\centering
	\captionsetup[subfigure]{labelformat=empty}
	\begin{subfigure}[b]{0.495\textwidth}
		\centering
		\includegraphics[width=\textwidth]{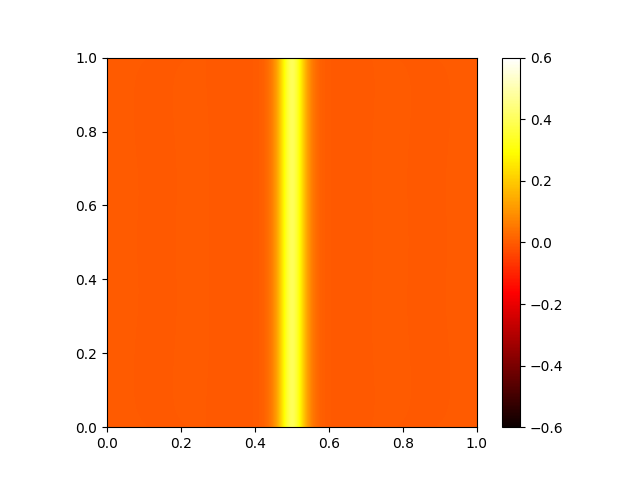}
	\end{subfigure}
	\begin{subfigure}[b]{0.495\textwidth}  
		\centering 
		\includegraphics[width=\textwidth]{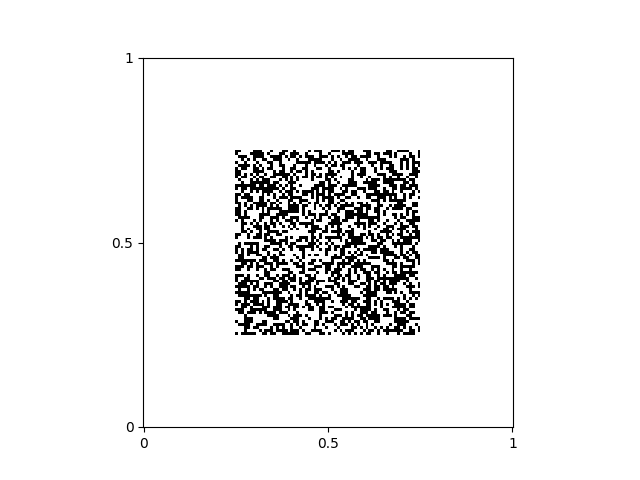}
	\end{subfigure}
	\caption[]
	{\small Real part of the incident beam $u_\mathrm{inc}$ (left) and Kerr coefficient with values $0$ (white) and $0.85$ (black) for Example~2.} 
	\label{fig:coeffBeam}
\end{figure*}
For the second example, we go beyond the above theory and consider an inhomogeneous boundary condition in~\eqref{eq:NLH}, which is equivalent to adding a boundary source $g$. In particular, we choose the wave number $k=30$ as well as the right-hand side $f$ and the boundary term $g$ that correspond to the incident beam
\begin{equation*}\label{eq:beam}
u_\mathrm{inc}(x) = \frac{0.8\exp(-\iu k (0.5 x_1 - 0.25)}{\cosh(50x_1 - 25)+1},
\end{equation*}
which is illustrated in Figure~\ref{fig:coeffBeam}~(left). 
More precisely, we take
\begin{equation*}
f = - \Delta u_\mathrm{inc} - k^2 u_\mathrm{inc}\quad\text{in }D, \qquad g = \nabla u_\mathrm{inc} \cdot \nu + \iu k u_\mathrm{inc}\quad\text{on }\Gamma.
\end{equation*}
As above, we choose $A$ to be piecewise constant on the mesh $\calT_\eta$, where values on each element are obtained as independent and uniformly distributed random numbers within the intervals $[0.2,1]$. Further, $A$ is set to $1$ in $D\setminus[0.25,0.75]^2$; cf.~Assumption~\ref{a:coeff}. The Kerr coefficient is chosen as depicted in Figure~\ref{fig:coeffBeam}~(right) with values $0$ (white) and $0.85$ (black), and $n \equiv 1$. Note that the nonlinearity is active in $[0.25,0.75]^2$ in this case. We again compute the corresponding reference solution iteratively on the mesh $\calT_h$. The threshold for the (relative) residual is again $10^{-12}$, which is reached after $20$ iterations. The reference solution is shown in Figure~\ref{fig:errSol}~(right) and shows scattering effects that appear due to the heterogeneity in $A$ and the nonlinear Kerr term. 

\begin{figure}
	\begin{center}
		\scalebox{0.69}{
			\begin{tikzpicture}
				
			\begin{axis}[%
			width=3.0in,
			height=2.8in,
			scale only axis,
			at={(0.758in,0.481in)},
			scale only axis,
			xmin=0.0066,
			xlabel={\large mesh size $H$},
			xmax=0.6,
			xminorticks=true,
			ymode=log,
			xmode=log,
			ymin=1.e-02,
			ymax=6.6,
			yminorticks=true,
			axis background/.style={fill=white},
			legend style={legend cell align=left, align=left,
				draw=white!15!black},
			legend pos = south east,
			]
			
			\addplot [color=myBlue, mark=square, line width=1.5pt, mark size=4.0pt]
			table[row sep=crcr]{%
				0.5	1.07812846722511\\
				0.25	1.21611281274903\\
				0.125	3.64712471309059\\
				0.0625	0.482304108752061\\
				0.03125	0.0892937572721887\\
				0.015625	0.0321656628915674\\
				0.0078125	0.0148363283948829\\
			};
			\addlegendentry{$\mathrm{LOD}_\mathrm{ad}$}
			
			\addplot [color=myGreen, mark=triangle, line width=1.5pt, mark size=4.0pt]
			table[row sep=crcr]{%
				0.5	1.0782571229209\\
				0.25	1.21183836206265\\
				0.125	3.11753634510907\\
				0.0625	0.486673966084616\\
				0.03125	0.0953446583204031\\
				0.015625	0.0342477583776324\\
				0.0078125	0.0154504352232345\\
			};
			\addlegendentry{$\mathrm{LOD}_\infty$}
			
			\addplot [color=myOrange, mark=diamond, line width=1.5pt, mark size=4.0pt]
			table[row sep=crcr]{%
				0.5	1.07812846722511\\
				0.25	1.21611281274903\\
				0.125	5.17852259811404\\
				0.0625	0.487589287120045\\
				0.03125	0.0891899576903648\\
				0.015625	0.0320688813828965\\
				0.0078125	0.0147005327782282\\
			};
			\addlegendentry{$\mathrm{LOD}_0$}
			
			\addplot [color=myRed, mark=o, line width=1.5pt, mark size=4.0pt]
			table[row sep=crcr]{%
				0.5	0.989084548183453\\
				0.25	0.978107637200998\\
				0.125	0.974190196165165\\
				0.0625	0.949289382403951\\
				0.03125	0.737628026248712\\
				0.015625	0.489017150256851\\
				0.0078125	0.203290704145362\\
			};
			\addlegendentry{FEM}
			
			\addplot [color=black, dashed,line width=1.5pt]
			table[row sep=crcr]{%
				0.5		0.75\\
				0.0078125	0.01171875\\
			};
			\addlegendentry{order $1$}

			\end{axis}
			\end{tikzpicture}%
		}
	\hfill
	\scalebox{.555}{
		\centering
		\includegraphics[width=\textwidth]{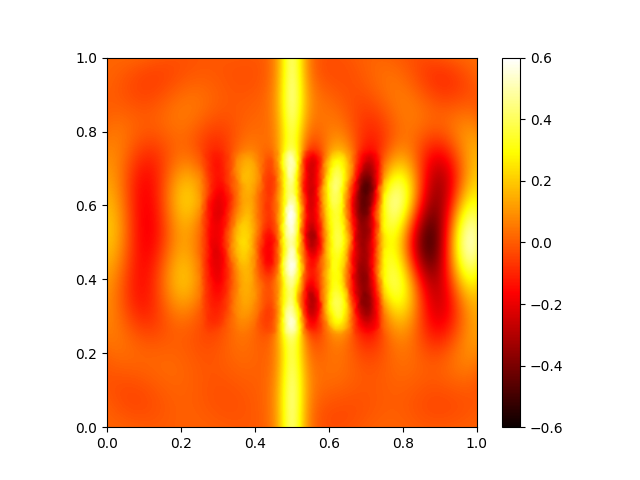}
	}
	\end{center}
	\caption[]
	{\small Relative energy errors of different iterative methods~(left) and real part of the reference solution~(right), both for Example~$2$.} 
\label{fig:errSol}
\end{figure}
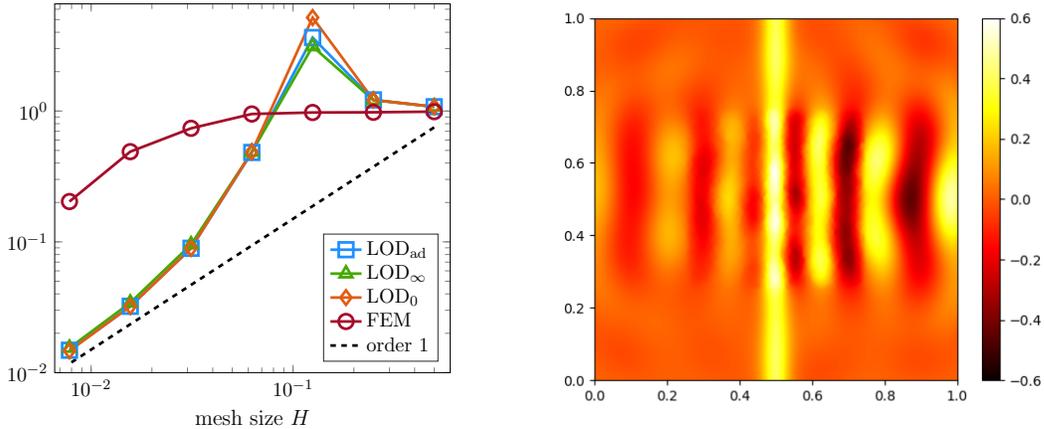

In Figure~\ref{fig:errSol}~(left) we present the smallest errors of LOD solutions with $\ell = 2$ for different values of $H$ within the first $20$ iterations. As above, we denote with $\mathrm{LOD}_\mathrm{ad}$~({\protect \tikz{ \draw[line width=1.pt, myBlue] (0,0) rectangle (0.18,0.18);}}) the iterative method which only updates the element correctors whose error indicators are larger than $50\%$ of the maximal value. Further, we show the errors for the iterative LOD solutions, where the correctors are only computed once ($\mathrm{LOD}_\infty$, {\protect \tikz{ \draw[line width=1.pt, myGreen] (0,0) --(0.18,0) -- (0.09,0.18) -- (0,0);}}) or always updated ($\mathrm{LOD}_0$, {\protect \tikz{ \draw[line width=1.pt, myOrange] (0.07,0) --(0.14,0.1) -- (0.07,0.2) -- (0.,0.1) -- (0.07,0);}}) and a classical finite element approximation~({\protect \tikz{ \draw[line width=1.pt, myRed] circle (0.6ex);}}). As above, we see a stagnation of the error curve of the finite element method if the scale of data oscillation is not resolved. 
Once the resolution condition $kH\lesssim 1$ is satisfied, {all} LOD curves show the predicted convergence and especially perform significantly better than the finite element approximation as expected.
Note that the increase in the two LOD errors for the three first mesh sizes is caused by the violation of the resolution condition and, hence, does not contradict the above theory.
In fact, we emphasize that similar peaks in the error curves are also observed for the linear Helmholtz equation if the resolution condition is not fulfilled, see~\cite{GalP15,BroGP17}.

Finally, we emphasize that the LOD curves with and without updates are very close in this example, since only very few iterations ($\leq 4$) are required to be close to a fixed point and the maximal values of the error indicator decrease heavily after the first iteration. The maximal update percentages {for $\mathrm{LOD}_\mathrm{ad}$}, however, are higher than in the first example, which is related to the support of the nonlinearity being larger in the second one. These values are given by 
$100,\,62.5,\,12.5,\,12.89,\,6.93,\,4.22,\,1.07$
(in \%) for $H = 2^{-1},\,2^{-2},\,2^{-3},\,2^{-4},\,2^{-5},\,2^{-6},\,2^{-7}$. 
\subsection{Example 3: tolerance} 
Our third example is dedicated to an investigation of the behavior of the approximate solutions when only the tolerance \texttt{tol} is changed. As mentioned above, an update strategy based on the maximal value of the error estimators in every step is employed in Example~1 and~2, which generally leads to very good results with few updates. However, since we considered a fixed value of \texttt{tol} in the theoretical part above, we now present an example with this update strategy to indicate the influence of \texttt{tol}. Therefore, let $k = 15$, $f \equiv 100$, $A \equiv 1$, $\eps \equiv 0.3\,\id_{[0.15,0.85]^2}$, and $n$ as given in Figure~\ref{fig:exp3}~(right). 
As before, a reference solution is computed on the mesh $\calT_h$ with an iterative finite element strategy. The threshold for the (relative) residual is $10^{-12}$, which is reached after $22$ iterations. 

In Figure~\ref{fig:exp3}~(left), we present the development of the relative errors for LOD solutions with different tolerances for the fixed mesh size $H=2^{-3}$ and the localization parameter $\ell = 2$. The black dashed lines show the errors for the case where all correctors are always recomputed ($\mathtt{tol}=0$) and the errors for the case when correctors are never recomputed (which is the case for $\mathtt{tol}=2$). The plot shows how the errors decrease when the tolerance is lowered. Moreover, for the choice $\mathtt{tol} = 2^{-4}$ (and any smaller value of \texttt{tol}), the corresponding error curve is very close to the (optimal) curve for $\mathtt{tol}=0$. That is, the tolerance is then small enough such that the mesh size error dominates. This is in line with the error estimate derived in Theorem~\ref{t:errAdItLOD}. 

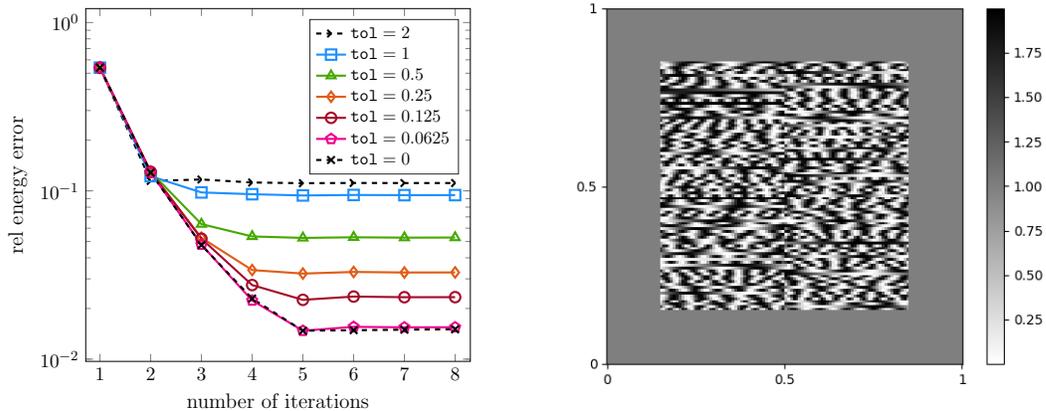
\begin{figure}
\begin{center}
	\scalebox{0.66}{
		\begin{tikzpicture}
			
			\begin{axis}[%
				width=3.05in,
				height=2.8in,
				scale only axis,
				at={(0.772in,0.481in)},
				scale only axis,
				xmin=0.72,
				xlabel={\large number of iterations},
				xmax=8.3,
				ymode=log,
				ymin=9.7e-3, 
				ymax=1.21,
				ylabel={\large rel energy error},
				yminorticks=true,
				axis background/.style={fill=white},
				legend style={legend cell align=left, align=left, draw=white!15!black},
				legend pos = north east
				]
				
				\addplot [color=black, mark=x, dashed, line width=1.2pt, mark size=3.0pt]
				table[row sep=crcr]{%
					1	0.539342422747006\\
					2	0.114521010265273\\
					3	0.116804702234896\\
					4	0.112111978394484\\
					5	0.11105526503728\\
					6	0.111389006935065\\
					7	0.1112662483236\\
					8	0.111279816148258\\
				};
				\addlegendentry{\small $\mathtt{tol} = 2$}
				
				\addplot [color=myBlue, mark=square, line width=1.2pt, mark size=3.0pt, mark options={solid}]
				table[row sep=crcr]{%
					1	0.539342422747006\\
					2	0.122211409426392\\
					3	0.0978402713795392\\
					4	0.0954979840783716\\
					5	0.0938868514508621\\
					6	0.0945067618177538\\
					7	0.0942839286318795\\
					8	0.0943402771874742\\
				};
				\addlegendentry{\small $\mathtt{tol} = 1$}
				
				\addplot [color=myGreen, mark=triangle, line width=1.2pt, mark size=3.0pt, mark options={solid}]
				table[row sep=crcr]{%
					1	0.539342422747006\\
					2	0.129495985769009\\
					3	0.0634504859461251\\
					4	0.053598510658141\\
					5	0.0525264114132362\\
					6	0.052968870340655\\
					7	0.0527186002920782\\
					8	0.0527894080699725\\
				};
				\addlegendentry{\small $\mathtt{tol} = 0.5$}
				
				\addplot [color=myOrange, mark=diamond, line width=1.2pt, mark size=3.0pt, mark options={solid}]
				table[row sep=crcr]{%
					1	0.539342422747006\\
					2	0.129205959837784\\
					3	0.0526001128043829\\
					4	0.0338348315973999\\
					5	0.0322500906368631\\
					6	0.0329697588810212\\
					7	0.0327260647318199\\
					8	0.0327587767769967\\
				};
				\addlegendentry{\small $\mathtt{tol} = 0.25$}
				
				\addplot [color=myRed, mark=o, line width=1.2pt, mark size=3.0pt]
				table[row sep=crcr]{%
					1	0.539342422747006\\
					2	0.129760395594983\\
					3	0.0519742213711254\\
					4	0.0275546691585753\\
					5	0.0225127908738953\\
					6	0.0235483985873276\\
					7	0.0233417997758009\\
					8	0.0233569801282098\\
				};
				\addlegendentry{\small $\mathtt{tol} = 0.125$}
				
				\addplot [color=magenta, mark=pentagon, line width=1.2pt, mark size=3.0pt, mark options={solid}]
				table[row sep=crcr]{%
					1	0.539342422747006\\
					2	0.128445528723396\\
					3	0.0479093837669409\\
					4	0.0223124173042601\\
					5	0.0147546219034979\\
					6	0.0155638386852934\\
					7	0.0154594931667194\\
					8	0.0154755783187324\\
				};
				\addlegendentry{\small $\mathtt{tol} = 0.0625$}
				
				\addplot [color=black, dashed, mark=x,  line width=1.2pt, mark size=3.0pt, mark options={solid}]
				table[row sep=crcr]{%
					1	0.539342422747006\\
					2	0.128362357747965\\
					3	0.0476497792168072\\
					4	0.0229070792854684\\
					5	0.0147589083017931\\
					6	0.0148592166117832\\
					7	0.0149676353932006\\
					8	0.0150419804455924\\
				};
				\addlegendentry{\small $\mathtt{tol} = 0$}
				
			\end{axis}
		\end{tikzpicture}%
	}
	\hfill
	\scalebox{.545}{
		\centering
		\includegraphics[width=\textwidth]{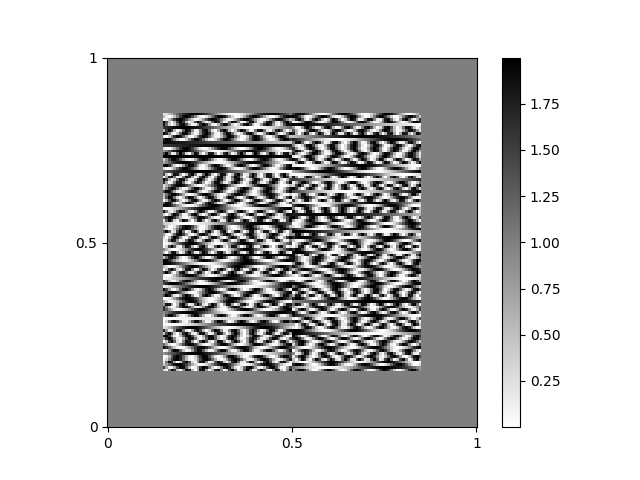}
	}
\end{center}
\caption[]
{\small Relative energy errors for the LOD method with respect to number of iterations for $H=2^{-3}$, $\ell = 2$ and different values of \texttt{tol}~(left) and coefficient $n$~(right), both for Example~$3$.} 
\label{fig:exp3}
\end{figure}

\section{Conclusion}\label{s:concl}

Within this work, we proposed and analyzed an iterative multiscale method for a~heterogeneous Helmholtz problem involving a Kerr-type nonlinearity. The method iteratively constructs (coarse) multiscale spaces that are partially updated in each iteration based on an appropriate error indicator. The approach allows for variations in the Kerr coefficient, the refractive index, and the diffusion coefficient. 
We have proved well-posedness and convergence estimates of the method for arbitrarily rough coefficients under mild resolution conditions on the mesh size and appropriate bounds for the nonlinearity. In particular, we have quantified the influence of localization, linearization, and partial recalculations of multiscale spaces in each iteration.
The presented numerical examples have confirmed the theoretical findings and in particular show that small updates in every iteration and moderate choices of the oversampling parameter are already sufficient to obtain reasonable approximations. 

In general, adaptive iterative multiscale approximations open up a new perspective on the design of multiscale methods for nonlinear problems. In future research, this idea may be transferred to other types of problems such as, e.g., quasilinear diffusion-type problems or time-dependent quasilinear wave problems.
\section*{Acknowledgments} 
We thank the anonymous reviewers for their remarks, which helped to improve the paper.
Roland Maier gratefully acknowledges support by the G\"oran Gustafsson Foundation for Research in Natural Sciences and Medicine. %
Barbara Verf\"urth's work is funded by the German Research Foundation (DFG) -- Project-ID 258734477 -- SFB 1173 as well as by the Federal Ministry of Education and Research (BMBF) and the Baden-W\"urttemberg Ministry of Science as part of the Excellence Strategy of the German Federal and State Governments.
\appendix
\section{Proofs of the results in Section~\ref{ss:adaptive}}
In this appendix, we prove Lemma~\ref{l:erroindic} and Theorem~\ref{t:errAdItLOD}.
We use the same notation as in Section~\ref{ss:adaptive}.

\subsection{Proof of Lemma~\ref{l:erroindic}}

Recall that $E(\calC_{\Phi, T}^\ell, \calC_{\Psi, T}^\ell)$ is an indicator in the sense that it yields an upper bound on the error $\calC_{\Phi, T}^\ell-\calC_{\Psi, T}^\ell$ without computing $\calC_{\Phi, T}^\ell$ itself.  
The indicator consists of two main factors: 
\begin{itemize}
\item[(i)] the error in the nonlinear coefficient
\begin{equation*} 
\|n\varepsilon(|\Psi|^2-|\Phi|^2)\|_{L^\infty(K)}
\end{equation*} 
for which only $\Phi$ and $\Psi$ themselves are required, but not the corresponding correctors and 
\item[(ii)] the factor 
\begin{equation*} 
\max_{v\vert_T\,:\, v\in V_H}\frac{\|\chi_T v-\calC_{\Psi, T}^\ell v\|_{0, K}^2}{\|v\|_{0,T}^2}
\end{equation*} 
which can be pre-computed when calculating $\calC_{\Psi, T}^\ell$.
This requires the solution of a small $L^2$-eigenvalue problem and in particular, the factor itself is a coarse-scale quantity so that the demanded storage is negligible. 
\end{itemize}

We emphasize that in practice the functions $\Phi$ and $\Psi$ in the error indicator will be iterative multiscale approximations and thus fully discrete objects.
This justifies the use of the $L^\infty$-norm here.
Throughout the paper and especially in Section~\ref{s:NLH} we avoided the $L^\infty$-norm because of the low regularity setting.
One can also rewrite the error indicator and the proof below with the Nirenberg-type inequality such that it holds for general functions $\Phi, \Psi\in H^1(\Omega)\setminus L^\infty (\Omega)$.
The rewritten error indicator involves $L^3$-norms in its two main factors analyzed above which we did not consider for implementational reasons.

\begin{proof}[Proof of Lemma~\ref{l:erroindic}]
We abbreviate $z:=(\calC_{\Phi, T}^\ell-\calC_{\Psi, T}^\ell)v_H$ and recall that $z \in \W(\Nb^\ell(T))$. 
Lemma~\ref{l:infsupW} and \eqref{eq:elementcorrector} therefore yield
\begin{align*}
\gamma \|z\|_{1,k}^2\leq \calB_{\mathrm{lin}}(\Phi; z, z)
& = \calB_{\mathrm{lin}, T}(\Phi; v_H, z)-\calB_{\mathrm{lin}, T}(\Psi; v_H, z)\\&\qquad+\calB_{\mathrm{lin}, \Nb^\ell(T)}(\Psi; \calC_{\Psi, T}^\ell v_H, z)- \calB_{\mathrm{lin}, \Nb^\ell(T)}(\Phi; \calC_{\Psi, T}^\ell v_H, z)\\
&\leq k\|n\varepsilon(|\Phi|^2-|\Psi|^2) (\chi_T v_H-\calC_{\Psi, T}^\ell v_H)\|_{0, \Nb^\ell(T)}\|z\|_{1,k}.
\end{align*}
After dividing by $\|z\|_{1,k}$ and taking the square on both sides, we further obtain
\begin{align*}
\gamma^2\|z\|_{1,k}^2&\leq k^2\|n\varepsilon(|\Phi|^2-|\Psi|^2) (\chi_T v_H-\calC_{\Psi, T}^\ell v_H)\|_{0, \Nb^\ell(T)}^2\\
&\leq \sum_{K\in \Nb^\ell(T)} k^2 \|n\varepsilon(|\Phi|^2-|\Psi|^2)\|_{L^\infty(K)}^2\|(\chi_T v_H-\calC_{\Psi, T}^\ell v_H)\|_{0, K}^2\\
&\leq  \sum_{K\in \Nb^\ell(T)} \|n\varepsilon(|\Phi|^2-|\Psi|^2)\|_{L^\infty(K)}^2\Bigl(\max_{w|_T\,:\, w\in V_H}\frac{\|(\chi_T w-\calC_{\Psi, T}^\ell w)\|_{0, K}^2}{\|w\|_{0,T}^2}\Bigr)k^2\|v_H\|_{0,T}^2\\
&\leq E(\calC_{\Phi, T}^\ell, \calC_{\Psi, T}^\ell)^2\,\|v_H\|_{1,k,T}^2,
\end{align*}
which shows \eqref{eq:indicloc}.

Due to the definition of $\calC_{\Phi}^\ell$ and $\calC_{\Psi}^\ell$ as sums of element correctors with support only in $\Nb^\ell(T)$, we then deduce that there exists a constant $C_{\mathrm{ol}}$, independent of $H$ and $\ell$ such that 
\begin{equation*}
\begin{aligned}
\|(\calC_{\Phi}^\ell-\calC_{\Psi}^\ell)v_H\|_{1,k}^2&\leq C_{\mathrm{ol}}^2 \ell^{d}\sum_{T\in \calT_H}\|(\calC_{\Phi, T}^\ell-\calC_{\Psi,T}^\ell)v_H\|_{1,k}^2\\
&\leq C_{\mathrm{ol}}^2\ell^{d}\gamma^{-2}\big(\max_{T\in \calT_H}E(\calC_{\Phi, T}^\ell, \calC_{\Psi, T}^\ell)\big)^2\,\|v_H\|_{1,k}^2
\end{aligned}
\end{equation*}
and, hence, \eqref{eq:indicglobal}. 
\end{proof}

\subsection{Proof of Theorem~\ref{t:errAdItLOD}}
Inspired by the proof of Theorem~\ref{t:errItLOD}, we define for fixed $m \in \N$ an auxiliary solution $\tilde{w}^m\in H^1(D)$ as the solution to
\begin{equation}\label{eq:auxsol}
\calBl(\uHT^{m-1};\tilde w^m,v) = (f,v)
\end{equation}
for all $v \in H^1(D)$. 
The next proposition quantifies the error $\tilde{w}^m-\uHT^m$. In particular, we first show that the LOD problem \eqref{eq:itmssol} in Algorithm~\ref{alg:adaptive} is well-posed. Note that this does not directly follow from Lemma \ref{l:stabapproxLOD} because not all element correctors are computed anew.

\begin{proposition}\label{prop:adapitLOD}
Let $m\geq 1$ be fixed and let $\uHT^{m-1}$ satisfy \eqref{eq:stabPhicontr}.
Further, let $\delta(k)$ denote the inf-sup constant with respect to $\calBl(\uHT^{m-1};\cdot,\cdot)$ in $H^1(D)$ as quantified in Lemma~\ref{l:infsup}. 
Assume that~\eqref{eq:oversampl} and~\eqref{eq:condtol} are fulfilled.
Then, the following inf-sup condition is satisfied,
\begin{equation*}
\adjustlimits \inf_{v_H\in V_H}\sup_{w_H\in V_H}\frac{\Re\,\calBl(\uHT^{m-1};(1-\tilde{\calC}_{m-1}^\ell)v_H, (1-\tilde{\calC}_{m-1}^{\ell,*})w_H)}{\|(1-\tilde{\calC}_{m-1}^{\ell})v_H\|_{1,k}\,\|(1-\tilde{\calC}_{m-1}^{\ell,*})w_H\|_{1,k}}\geq \frac{\delta(k)\gamma^2}{18C_\mathrm{int}^2C_\calB^2}.
\end{equation*}
In particular, $\uHT^m$ is well-defined.
Furthermore, we have
\begin{equation}\label{eq:erradapLODm}
\|\tilde{w}^m-\uHT^m\|_{1,k}\leq 2\gamma^{-1}\,C_\mathrm{int}\big(H+C_\calB C_{\mathrm{loc}}\,\ell^{d/2}\beta^\ell+C_\calB C_\mathrm{stab}(k)C_{\mathrm{ol}}\,\ell^{d/2} \,\gamma^{-1}\, \mathtt{tol}\big)\|f\|_0.
\end{equation}
\end{proposition}

\begin{proof}
\emph{Step 1 (inf-sup constant)}:
Let $v_H\in V_H$ be arbitrary but fixed.
Due to the continuous inf-sup condition~\eqref{eq:infsup}, there exists $w\in H^1(D)$ with $\|w\|_{1,k}=1$ such that
\begin{equation*}
\Re\,\calBl(\uHT^{m-1};(1-\calC_{\scalebox{.65}{$\uHT^{m-1}$}})v_H,w)\geq \delta(k)\,\|(1-\calC_{\scalebox{.65}{$\uHT^{m-1}$}})v_H\|_{1,k}.
\end{equation*}
We set $w_H=\IH w$ and deduce
\begin{equation}\label{eq:proofAuxsol}
\begin{aligned}
\Re\,\calBl(\uHT^{m-1};(1-\tilde{\calC}_{m-1}^\ell&) v_H, (1-\tilde{\calC}_{m-1}^{\ell,*})w_H)\\
&=\Re\,\calBl(\uHT^{m-1};(1-\calC_{\scalebox{.65}{$\uHT^{m-1}$}}^\ell) v_H, (1-\tilde{\calC}_{m-1}^{\ell,*})w_H)\\
&\qquad-\Re\,\calBl(\uHT^{m-1};(\tilde{\calC}_{m-1}^\ell-\calC_{\scalebox{.65}{$\uHT^{m-1}$}}^\ell)v_H, (1-\tilde{\calC}_{m-1}^{\ell,*})w_H)\\
&=\Re\,\calBl(\uHT^{m-1};(1-\calC_{\scalebox{.65}{$\uHT^{m-1}$}}) v_H, (1-\tilde{\calC}_{m-1}^{\ell,*})w_H)\\
&\qquad-\Re\,\calBl(\uHT^{m-1};(\calC_{\scalebox{.65}{$\uHT^{m-1}$}}^\ell-\calC_{\scalebox{.65}{$\uHT^{m-1}$}})v_H, (1-\tilde{\calC}_{m-1}^{\ell, *})w_H)\\
&\qquad-\Re\,\calBl(\uHT^{m-1};(\tilde{\calC}_{m-1}^\ell-\calC_{\scalebox{.65}{$\uHT^{m-1}$}}^\ell)v_H, (1-\tilde{\calC}_{m-1}^{\ell,*})w_H)\\
&=\Re\,\calBl(\uHT^{m-1};(1-\calC_{\scalebox{.65}{$\uHT^{m-1}$}}) v_H, w)\\
&\qquad-\Re\,\calBl(\uHT^{m-1};(\calC_{\scalebox{.65}{$\uHT^{m-1}$}}^\ell-\calC_{\scalebox{.65}{$\uHT^{m-1}$}})v_H, (1-\tilde{\calC}_{m-1}^{\ell, *})w_H)\\
&\qquad-\Re\,\calBl(\uHT^{m-1};(\tilde{\calC}_{m-1}^\ell-\calC_{\scalebox{.65}{$\uHT^{m-1}$}}^\ell)v_H, (1-\tilde{\calC}_{m-1}^{\ell,*})w_H).
\end{aligned}
\end{equation}
For the last equality, we wrote $1-\tilde{\calC}_{m-1}^{\ell,*}=1-\tilde{\calC}_{m-1}^*+\tilde{\calC}_{m-1}^*-\tilde{\calC}_{m-1}^{\ell, *}$ and used the orthogonality of $(1-\calC_{\scalebox{.65}{$\uHT^{m-1}$}})V_H$ and $\W$ as well as $(1-\tilde{\calC}_{m-1}^*)\IH w = (1-\tilde{\calC}_{m-1}^*)w$.
Here, $\tilde{\calC}_{m-1}$ is defined as $\tilde{\calC}_{m-1}^\ell$ with $\ell=\infty$, i.e., on global patches.
We can now use \eqref{eq:proofAuxsol}, the continuous inf-sup condition~\eqref{eq:infsup}, and the continuity of $\calBl$ to obtain 
\begin{equation}
\begin{aligned}\label{eq:proofAuxsol2}
\Re\,\calBl(&\uHT^{m-1};(1-\tilde{\calC}_{m-1}^\ell) v_H, (1-\tilde{\calC}_{m-1}^{\ell,*})w_H)\\
&\geq  \delta(k)\|(1-\calC_{\scalebox{.65}{$\uHT^{m-1}$}})v_H\|_{1,k}-C_\calB\|(\calC_{\scalebox{.65}{$\uHT^{m-1}$}}^\ell-\calC_{\scalebox{.65}{$\uHT^{m-1}$}})v_H\|_{1,k}\|(1-\tilde{\calC}_{m-1}^{\ell, *})w_H\|_{1,k}\\
&\qquad-C_\calB\|(\tilde{\calC}_{m-1}^\ell-\calC_{\scalebox{.65}{$\uHT^{m-1}$}}^\ell)v_H\|_{1,k}\|(1-\tilde{\calC}_{m-1}^{\ell, *})w_H\|_{1,k}.
\end{aligned}
\end{equation}
Lemma~\ref{l:trunccorrectorerr} yields
\begin{equation}\label{eq:prooftruncerr}
\|(\calC_{\scalebox{.65}{$\uHT^{m-1}$}}^\ell-\calC_{\scalebox{.65}{$\uHT^{m-1}$}})v_H\|_{1,k}\leq C_{\mathrm{loc}}\ell^{d/2}\beta^\ell \|v_H\|_{1,k}.
\end{equation}
Note that the adaptive algorithm ensures together with Lemma~\ref{l:erroindic} that 
\begin{equation}\label{eq:linerrorglobal}
\begin{aligned}
\|(\tilde{\calC}_{m-1}^\ell-\calC_{\scalebox{.65}{$\uHT^{m-1}$}}^\ell)v_H\|_{1,k}^2\leq C_{\mathrm{ol}}^2\ell^{d}\gamma^{-2}\,\texttt{tol}^2\|v_H\|_{1,k}^2.
\end{aligned}
\end{equation}
Furthermore, we have norm equivalences
\begin{align*}
\|v_H\|_{1,k}=\|\IH((1-{\calC}_{\scalebox{.65}{$\uHT^{m-1}$}})v_H)\|_{1,k}&\leq C_\mathrm{int}\|(1-\calC_{\scalebox{.65}{$\uHT^{m-1}$}})v_H\|_{1,k},\\
\|v_H\|_{1,k}=\|\IH((1-\tilde{\calC}_{m-1}^\ell)v_H)\|_{1,k}&\leq C_\mathrm{int}\|(1-\tilde{\calC}_{m-1}^\ell)v_H\|_{1,k},
\end{align*}
as well as the estimate
\begin{align*}
\|(1-&\tilde{\calC}_{m-1}^{\ell})v_H\|_{1,k}\\
&\leq\|(1-\calC_{\scalebox{.65}{$\uHT^{m-1}$}})v_H\|_{1,k}+\|(\calC_{\scalebox{.65}{$\uHT^{m-1}$}}-\calC_{\scalebox{.65}{$\uHT^{m-1}$}}^{\ell})v_H\|_{1,k}+\|(\calC_{\scalebox{.65}{$\uHT^{m-1}$}}^{\ell}-\tilde{\calC}_{m-1}^{\ell})v_H\|_{1,k}\\
&\leq(\gamma^{-1}C_\calB+C_{\mathrm{loc}}\ell^{d/2}\beta^\ell+C_{\mathrm{ol}}\ell^{d/2}\gamma^{-1}\,\mathtt{tol})\|v_H\|_{1,k}
\\&\leq 3\gamma^{-1}C_\calB \|v_H\|_{1,k},
\end{align*}
where we have used \eqref{eq:oversampl} and \eqref{eq:condtol} in the last step.
With the same line of arguments we can also show that
\begin{align*}
\|(1-\tilde{\calC}_{m-1}^{\ell, *})w_H\|_{1,k}
&\leq 3\gamma^{-1}C_\calB\|w_H\|_{1,k}\leq 3\gamma^{-1}C_\calB C_{\mathrm{int}}\|w\|_{1,k}.
\end{align*}
Inserting \eqref{eq:prooftruncerr}--\eqref{eq:linerrorglobal} and these norm equivalences into \eqref{eq:proofAuxsol2}, we finally obtain
\begin{align*}
\Re\,\calBl(\uHT^{m-1};&(1-\tilde{\calC}_{m-1}^\ell) v_H, (1-\tilde{\calC}_{m-1}^{\ell, *})w_H)\\
&\geq \Bigl(\delta(k)\frac{\gamma^2}{9C_\calB^2C_{\mathrm{int}}^2}-C_\calB C_{\mathrm{int}}C_{\mathrm{loc}}\,\ell^{d/2}\beta^\ell-C_\calB C_{\mathrm{int}}C_{\mathrm{ol}}\,\ell^{d/2}\,\gamma^{-1}\,\texttt{tol}\Bigr)\\ &\hspace{4cm}\cdot\|(1-\tilde{\calC}_{m-1}^\ell)v_H\|_{1,k}\|(1-\tilde{\calC}_{m-1}^{\ell,*})w_H\|_{1,k},
\end{align*}
which finishes the proof of the inf-sup condition using the conditions~\eqref{eq:oversampl}--\eqref{eq:condtol}.

\emph{Step 2 (error estimate)}: We write $\uHT^{m}=(1-\tilde{\calC}_{m-1}^\ell)\IH \uHT^m$ and set $e:=\tilde w^m-\uHT^m$, where $\tilde w^m$ is the auxiliary solution defined in \eqref{eq:auxsol}. Further, let $e_H^\ell :=(1-\tilde{\calC}_{m-1}^\ell) \IH e$.
Here, $\tilde{\calC}_{m-1}$ is the corrector with $\ell=\infty$ as explained above.
Observe that $e-e_H^\ell = \tilde w^m-(1-\tilde{\calC}_{m-1}^\ell) \IH \tilde w^m\in \W$.
With Lemma~\ref{l:infsupW}, \eqref{eq:IH1}, \eqref{eq:prooftruncerr}, \eqref{eq:linerrorglobal}, and \eqref{eq:stabLin}, we obtain
\begin{equation}\label{eq:proofeeH}
\begin{aligned}
\gamma\,\|e-e_H^\ell\|_{1,k}^2&\leq\Re\,\calBl(\uHT^{m-1};\tilde w^m-(1-\tilde{\calC}_{m-1}^\ell)\IH \tilde w^m, e-e_H^\ell)\\
&=\Re\,(f,\overline{e-e_H^\ell})-\Re\, \calBl(\uHT^{m-1};(1-\calC_{\scalebox{.65}{$\uHT^{m-1}$}}^\ell)\IH \tilde w^m, e-e_H^\ell)\\
&\qquad-\Re\,\calBl(\uHT^{m-1};(\calC_{\scalebox{.65}{$\uHT^{m-1}$}}^\ell-\tilde{\calC}_{m-1}^\ell)\IH \tilde w^m, e-e_H)\\
&=\Re\,(f,\overline{e-e_H^\ell})-\Re\, \calBl(\uHT^{m-1};(\calC_{\scalebox{.65}{$\uHT^{m-1}$}}-\calC_{\scalebox{.65}{$\uHT^{m-1}$}}^\ell)\IH \tilde w^m, e-e_H^\ell)\\
&\qquad-\Re\,\calBl(\uHT^{m-1};(\calC_{\scalebox{.65}{$\uHT^{m-1}$}}^\ell-\tilde{\calC}_{m-1}^\ell)\IH \tilde w^m, e-e_H)\\
&\leq C_\mathrm{int}\big(H+C_\calB C_{\mathrm{loc}}\ell^{d/2}\beta^\ell+ C_\calB C_\mathrm{stab}(k)C_{\mathrm{ol}}\,\ell^{d/2}\,\gamma^{-1}\,\texttt{tol}\big)\|f\|_0\,\|e-e_H^\ell\|_{1,k}.
\end{aligned}
\end{equation}
From the definitions of $\tilde w^m$ and $\uHT^m$, we deduce the following Galerkin orthogonality
\begin{equation*} 
\calBl(\uHT^{m-1}; e, w_H)=0
\end{equation*}
for all $w_H\in (1-\tilde{\calC}_{m-1}^{\ell, *})V_H$. 
Let $z_H^\ell\in (1-\tilde{\calC}_{m-1}^{\ell, *})V_H$ be the unique solution to the dual problem
\begin{equation*}
\calBl(\uHT^{m-1};v_H, z_H^\ell)=(v_H, e_H^\ell)_{1,k}
\end{equation*}
for all $v_H\in (1-\tilde{\calC}_{m-1}^\ell)V_H$, which is well-posed by the inf-sup condition that we have proved in step 1.
We observe  that $z_H^\ell=(1-\tilde{\calC}_{m-1}^{\ell, *})\IH z_{H}^\ell$.
Due to the form of $e_H^\ell$ and the Galerkin orthogonality, we hence deduce 
\begin{align*}
\|e_H^\ell\|_{1,k}^2&=\calBl(\uHT^{m-1};e_H^\ell, z_H^\ell)\\
&=\calBl(\uHT^{m-1};e_H^\ell-e, z_H^\ell)\\
&=\calBl(\uHT^{m-1};e_H^\ell-e, (1-\calC_{\scalebox{.65}{$\uHT^{m-1}$}}^{\ell,*})\IH z_H^\ell)+\calBl(\uHT^{m-1};e_H^\ell-e,(\calC_{\scalebox{.65}{$\uHT^{m-1}$}}^{\ell,*}-\tilde{\calC}_{m-1}^{\ell,*})\IH z_{H}^\ell)\\
&=\calBl(\uHT^{m-1};e_H^\ell-e, (\calC_{\scalebox{.65}{$\uHT^{m-1}$}}^*-\calC_{\scalebox{.65}{$\uHT^{m-1}$}}^{\ell,*})\IH z_H^\ell)\\
&\qquad+\calBl(\uHT^{m-1};e_H^\ell-e,(\calC_{\scalebox{.65}{$\uHT^{m-1}$}}^{\ell,*}-\tilde{\calC}_{m-1}^{\ell,*})\IH z_{H}^\ell)\\
&\leq C_\calB\,C_{\mathrm{int}}(C_{\mathrm{loc}}\,\ell^{d/2}\beta^\ell+C_{\mathrm{ol}}\ell^{d/2}\,\gamma^{-1}\,\texttt{tol})\,\|z_{H}^\ell\|_{1,k}\,\|e_H^\ell-e\|_{1,k}\\
&\leq C_\calB\,C_{\mathrm{int}}\, (C_{\mathrm{loc}}\,\ell^{d/2}\beta^\ell+C_{\mathrm{ol}}\ell^{d/2}\,\gamma^{-1}\,\texttt{tol})\,\frac{18C_\mathrm{int}^2C_\calB^2}{\delta(k)\gamma^2}\,\|e_H^\ell\|_{1,k}\,\|e_H^\ell-e\|_{1,k}.
\end{align*}
Using the condition \eqref{eq:oversampl} on $\ell$ and \eqref{eq:condtol}, we obtain 
\begin{equation*}
\|e_H^\ell\|_{1,k} \leq \|e_H^\ell-e\|_{1,k}.
\end{equation*}
Together with \eqref{eq:proofeeH} and an application of the triangle inequality, this concludes the proof.
\end{proof}
This Proposition is now used  inductively to prove Theorem~\ref{t:errAdItLOD}.
\begin{proof}[Proof of Theorem~\ref{t:errAdItLOD}]
We proceed similarly as in the proof of Theorem~\ref{t:errItLOD}.
Let $m\geq 1$ be fixed and define the auxiliary solution $\tilde{w}^m\in H^1(D)$ as the solution to~\eqref{eq:auxsol}. 
Note that due to~\eqref{eq:boundRHSadap} and by Proposition~\ref{prop:auxlinpb} (as well as its adapted version for the series of multiscale solutions), we have the following stability estimates,  
\begin{align*}
\|\tilde{u}^{m-1}_{H}\|_{1,k} &\leq \frac{18\,C_\mathrm{int}^2 C_\calB^2}{k\delta(k)\gamma^2}\,\|f\|_0,\\
\|u^m\|_{1,k} &\leq C_\mathrm{stab}(k)\,\|f\|_0,\\
\|\tilde{w}^m\|_{1,k} &\leq C_\mathrm{stab}(k)\,\|f\|_0.
\end{align*}
Further, the error $e^m = u^m-\tilde{w}^m$ solves
\begin{equation*}
\begin{aligned}
\calBl(u^{m-1};e^m,v) &= \calBl(\uHT^{m-1};w^m,v) - \calBl(u^{m-1};w^m,v)\\
&= (k^2\eps(|u^{m-1}|^2 - |\uHT^{m-1}|^2) w^m,v)
\end{aligned}
\end{equation*}
for all $v \in H^1(D)$. 
As in the proof of Theorem~\ref{t:errItLOD}, we obtain
\begin{align*}
\|e^m\|_{1,k} &\leq \vartheta\,\|\uHT^{m-1} - u^{m-1}\|_{1,k}.
\end{align*}
Combining this with Proposition~\ref{prop:adapitLOD}, we deduce
\begin{align*}
\|\uHT^m-u^m\|_{1,k}&\leq \vartheta\|\uHT^{m-1}-u^{m-1}\|_{1,k} \\&\quad+2\gamma^{-1}C_\mathrm{int}\big(H+C_\calB C_{\mathrm{loc}}\,\ell^{d/2}\beta^\ell+C_\calB C_\mathrm{stab}(k)C_{\mathrm{ol}}\,\ell^{d/2}\,\gamma^{-1}\, \mathtt{tol}\big)\|f\|_0\\
&\leq \vartheta^m\|\uHT^0-u^0\|_{1,k}\\*&\quad+2\gamma^{-1}C_\mathrm{int}\big(H+C_{\mathrm{loc}}C_\calB\,\ell^{d/2}\beta^\ell+ C_\calB C_\mathrm{stab}(k)C_{\mathrm{ol}}\,\ell^{d/2}\,\gamma^{-1}\, \mathtt{tol}\big)\|f\|_0\sum_{n\in\mathbb{N}}\vartheta^n\\
&\leq \vartheta^m\|\uHT^0-u^0\|_{1,k}\\*&\quad+\frac{2}{1-\vartheta}\gamma^{-1}C_\mathrm{int}\big(H+C_\calB C_{\mathrm{loc}}\,\ell^{d/2}\beta^\ell+C_\calB C_\mathrm{stab}(k)C_{\mathrm{ol}}\,\ell^{d/2}\,\gamma^{-1}\, \mathtt{tol}\big)\|f\|_0,
\end{align*}
which concludes the proof.
\end{proof}

\end{document}